%% file: PAPER.tex
\documentclass[psamsfonts]{amsart}

\input{preamble}

\title{Approximation by simple structures}

\author[K. Ellman-Aspnes]{Katie Ellman-Aspnes}\thanks{The author was supported by Ramsey's NSF grant DMS-2246992 and NSF CAREER award DMS-2442011}

\date{\today}

\begin{document}

\begin{abstract}
In her graduate thesis, Gwyneth Harrison-Shermoen developed an abstract framework for approximating structures of a given theory, which generalized the notion of smoothly approximable structures. At the time it was unclear how broadly the framework applied due to a lack of examples. In this paper we will show that with slight modification to the framework, several well-understood $\mathrm{NSOP}_1$ theories are in fact approximable by simple structures. We discuss certain structural properties of the limit structure that can be obtained from the approximation and introduce a notion of limit dimension as a candidate for recovering Kim-independence in the limit structure. We also provide some examples of structures that are not approximable by structures of finite Morley rank.
\end{abstract}

\maketitle

\setcounter{tocdepth}{1}
\tableofcontents

\section{Introduction}
Two central features in the development of simple theories are the well-behaved nature of non-forking independence and the rich class of examples provided by smoothly approximable structures. Much work has been done in recent years on generalizing the former to the $\nsop_1$ context via the development of Kim-independence as an analogue of non-forking independence in the simple context. This raises the natural question as to whether there is a corresponding analogue of smooth approximability.

Smoothly approximable structures are homogeneous structures which can be written as the union of a chain of finite homogeneous substructures, with the key point being that nice properties from the finite homogeneous structures pass upwards to the limit. The first natural example of a similar construction in the $\nsop_1$ context is the Granger example $T_\infty$, the two-sorted theory of an infinite-dimensional vector space over an algebraically closed field with a nondegenerate symmetric or alternating bilinear form, detailed in \cite[Chapter 12]{granger1999stability}. The construction of this example directly parallels that of smoothly approximable structures, with Granger's approximation of $T_\infty$ taking the infinite-dimensional vector space to be the union of a chain of finite-dimensional subspaces of the approximated model. The theory of the limit, $T_\infty$, is $\nsop_1$ \cite[Corollary 6.4]{chernikov2016model}, and the theories of the approximating structures are simple (even stable) \cite[Theorem 10.2.3]{granger1999stability}. 

Granger also defined the notion of \textit{$\Gamma$-non-forking independence} \cite[Definition 12.2.1]{granger1999stability} which in the approximated structure is the limit of the corresponding notion in the sequence of approximating structures, and showed that in the finite-dimensional case, $\Gamma$-non-forking independence coincides with non-forking independence \cite[Proposition 10.2.2]{granger1999stability}. In the infinite-dimensional ($\nsop_1$) case, $\Gamma$-non-forking independence \textit{strengthens} Kim-independence; however they are not equal because $\Gamma$-non-forking-independence does not satisfy local character in uncountable models of $T_\infty$ \cite[Theorem 12.2.2]{granger1999stability}.

Granger's initial construction was further developed into a more general axiomatic framework by Harrison-Shermoen in \cite{harrison2013independence}. This framework extended the original ideas from Granger's approximating sequences to a directed system of approximating substructures, and defined the notion of \textit{limit independence} as the limit of non-forking independence in the directed system \cite[Section 3]{harrison2013independence}. Limit independence maintained the nice relationship between independence in the approximating and limit structures exhibited by $\Gamma$-non-forking independence. However, while the latter was defined in the specific context of $T_\infty$, limit independence is determined solely through non-forking independence and therefore can be understood with respect to any sufficiently saturated approximated structure. This general framework therefore suggested the potential for a much broader application of approximation within $\nsop_1$ theories. Despite this apparent generality, however, $T_\infty$ remained the only known example of a theory approximable in the broader framework. It was therefore unclear whether approximability was a relatively common or rare property of $\nsop_1$ theories.

We show that with some slight adjustments, the main well-known $\nsop_1$ examples do fit into this framework. In Section \ref{sec:approx} we will detail the original definition and results from \cite{harrison2013independence} along with the necessary adaptations to the framework, and in Section \ref{sec:examples} we will detail two additional $\nsop_1$ examples: $T^*_{\feq}$ and $\omega$-free PAC fields. These, along with $T_\infty$, are three very different types of $\nsop_1$ theories, which suggests that approximability is in fact a relatively common property of $\nsop_1$.

In Section \ref{sec:nonex}, we consider the question of when approximability is \textit{not} feasible, particularly in the context of generalizing the $T_\infty$ example to other linear structures. Borrowing arguments from the pseudofinite context \cite{macpherson2025omega}, we give some examples of theories whose models are not approximable by structures of finite Morley rank. This includes theories which are $\nsop_1$, so while approximability by finite-dimensional structures is a common feature of $\nsop_1$ theories, it is not universal.

In the final section we seek to address the question of recovering Kim-independence in the limit. The analysis of the approximated structures in \cite{harrison2013independence} relied heavily on a notion of limit independence $\ind^{\lim}$ (Definition \ref{def:limind}). While $\ind^{\lim}$ may be non-trivial, in the case where the approximating theories are simple it will always satisfy base monotonicity, meaning it cannot be used to recover Kim-independence in strictly $\nsop_1$ theories. Using some of the same ideas as in \cite{dobrowolski2023sets}, in Section \ref{sec:dim} we turn instead to a notion of limit \textit{dimension} with the goal of recovering Kim-independence in the approximated structure from some notion of dimension within the approximation itself. We will use the example $T^*_{\feq}$ (detailed in Section \ref{sec:examples}) as a case study, defining an explicit notion of dimension that does recover Kim-independence in the limit for $T^*_{\feq}$. We will then propose an initial attempt at a more abstract axiomatic notion of dimension that we hope will generalize the notion for $T^*_{\feq}$ to any approximable theory.

\section{Approximation} \label{sec:approx}

We begin by establishing the approximation framework from \cite{harrison2013independence}.

Given a theory $T$, we take $\M$ to be a model of $T$ and $\cH$ a directed system of substructures of $\M$. In the following definitions, $\bar{A}$ denotes $\acl^\M(A)$.

\begin{definition}\cite[Definition 3.2.1]{harrison2013independence}\label{def:limind}
    \begin{enumerate}
        \item For $\psi(x)\in\lang,a\in M$, we say that $\psi(a)$ is \textit{eventually true} in (or \textit{true in the limit of}) $\cH$ if there is some $\N\in\cH$ such that $a\in\N$;  and for $\N\subseteq \N'\in\cH$, $\N\models\psi(a)$. We define \textit{eventually false} analogously.
        \item Given a formula or finite partial type $\pi(x;y)\in\lang, b\in M$, and $C\subset M$, we say that $\pi(x; b)$ \textit{eventually forks} over $C$ if there is $\N\in\cH$ such that $b\in\N$ and for all $\N\subseteq \N'\in \cH$, $\pi(x; b)$ forks over $C\cap N'$ in $\N'$. We define \textit{eventual non-forking of a finite partial type} analogously.
        \item Given a (not necessarily complete) type $p(x)$ over $B$ and $C\subseteq B$, we say that $p(x)$ \textit{eventually forks} over $C$ if it contains some finite subtype that eventually forks over $C$, and that it \textit{eventually does not fork} over $C$ if every finite subtype eventually does not fork over $C$.
        \item For subsets $A,B$, and $C$ of $M$, we say that $A\ind^{\lim}_CB$ ($A$ is \textit{limit independent} or \textit{eventually independent} from $B$ over $C$) if for each finite tuple $a\in A$, there is $\N\in\cH$ such that $a\in\N$ and for all $\N\subseteq\N'\in\cH$,
        \[a\ind^{\N'}_{\bar{C}\cap N'}\bar{B}\cap N'\]
        where $\ind^{\N'}$ is non-forking independence as computed in $\N'$. That is, 
        \[\tp^{\N'}(a/\bar{B}\cup\bar{C}\cap N')\]
        does not fork over $\bar{C}\cap N'$ in $\N'$.
    \end{enumerate}
\end{definition}

\begin{definition}\label{def:approximation}
\cite[Definition 3.2.3]{harrison2013independence} Given $\M\models T$ and $\cH$ a directed system of substructures $\N$ of $\M$ we say $\cH$ \textit{approximates} $\M$ if conditions 1 through 5 hold, and identify optional conditions 6 and 7.
\begin{enumerate}
    \item \textbf{Covering}: $\bigcup\cH=\M$.
    \item \textbf{Automorphism Invariance}: $\cH$ is closed under automorphisms of $\M$.
    \item \textbf{Convergence of truth value}: For any $a\in\M$ and formula $\psi(x)\in\lang$, if $\M\models\psi(a)$ then $\psi(a)$ is eventually true in $\cH$. It follows that if $\M\models\neg\psi(a)$, then $\psi(a)$ is eventually false in $\cH$.
    \item \textbf{Stabilization of non-forking}: For every formula or finite partial type $\pi(x,y)$, tuple $b$, and set $C$ we have one of the following
    \begin{itemize}
        \item $\pi(x,b)$ eventually forks over $C$ in $\cH$
        \item $\pi(x,b)$ eventually does not fork over $C$ in $\cH$
    \end{itemize}
    \item \textbf{Strong finite character of $\nind^{\lim}$}: If $A\nind^{\lim}_CB$, then there are $\varphi(x,y,z)$ without parameters, $a\in A$, $b\in\bar{B}$, and $c\in\bar{C}$ such that $\M\models\varphi(a,b,c)$ and $\varphi(x,b,c)$ eventually forks over $\bar{C}$ in $\cH$.
    \item \textbf{Stabilization of algebraic closure}: If $\bar{A}=A$ then there is $\N_{\alg}\in\cH$ such that for all $\N\supseteq\N_{\alg}$, $A\cap N=\acl^\N(A\cap N)$.
    \item \textbf{Homogeneity}: For each $\N\in\cH$, $a,b\in M$, and $A\subset M$, if $a\equiv_A^\M b$ then $a\equiv_{A\cap N}^\N b$.
\end{enumerate}
\end{definition}

The key condition to note here is convergence of truth value. In \cite{macpherson2011definability} the authors considered asymptotic classes of finite structures and made an important distinction between the asymptotic theory and the limit theory of such a class. These do not necessarily agree in general, but in this setting the convergence of truth value condition guarantees that they coincide:

\begin{theorem}\label{thm:CTVultraproduct}
    Suppose $\cH$ is a directed system of substructures of $\M$ satisfying convergence of truth value. Then $\M$ is elementarily equivalent to an ultraproduct of elements of $\cH$.
\end{theorem}
\begin{proof}
    Let $\mathcal D$ be an ultrafilter on $\cH$ with the condition that for every $A\in\cH$, $\mathcal D$ contains the set $\{B\in\cH:A\subset B\}$. Let $P=\prod_{A\in\cH}A/\mathcal D$. We claim $P\models\Th(\M)$. 
    From the convergence of truth value condition, $\varphi\in\Th(\M)$ if and only if there exists an $\N_\varphi\in\cH$ such that for all $\N\in\cH$ with $\N\supseteq\N_\varphi$, $\N\models\varphi$. The choice of $\mathcal D$ guarantees that $\{B\in\cH:\N_\varphi\subset B\}\in\mathcal D$, so by \L o\^s's Theorem $P\models\varphi$. If $\varphi\notin\Th(\M)$, then an identical argument with $\neg\varphi$ guarantees that $P\not\models\varphi$. 
\end{proof}

In particular, this gives a means by which to determine when a structure cannot be approximated within this framework, an idea we will explore more in Section \ref{sec:nonex}.

In addition to the above definitions, Harrison-Shermoen proved that in the case where $\cH$ approximates $\M$ and every element of $\cH$ has a simple theory, then $\ind^{\lim}$ satisfies the following properties: invariance, monotonicity, base monotonicity, transitivity, normality, extension, finite character, strong finite character, and symmetry \cite[Theorem 3.3.2]{harrison2013independence}. Furthermore, if the approximation also satisfies optional conditions 6 and 7 and in each element of $\cH$ forking independence satisfies the Independence Theorem over algebraically closed sets, then $\ind^{\lim}$ also satisfies the Independence Theorem over algebraically closed sets \cite[Proposition 3.3.8]{harrison2013independence}.

It was noted in \cite{chernikov2016model} that this means in this situation $\ind^{\lim}$ strengthens Kim-independence and the approximated theory is $\nsop_1$. We will now take this one step further. 

\begin{lemma}\label{lem:localchar}
    (Local character of $\ind^{\lim}$) Assume that $\M$ is $\abs{T}^+$-saturated and that every element of $\cH$ has a simple theory. If $\kappa\geq\abs{T}^+$ is a regular cardinal, $\vect{A_i:i<\kappa}$ is an increasing continuous sequence of sets in $\M$ of size $<\kappa$, $A_{\kappa}=\bigcup_{i<\kappa}A_i$ and $\abs{A_\kappa}=\kappa$, then for any finite $d$ in $\M$ there is some $\alpha<\kappa$ such that $d\ind^{\lim}_{A_\alpha}A_\kappa$.
\end{lemma}
\begin{proof}
    Let $\vect{A_i:i<\kappa}$ as in the hypothesis and suppose for a contradiction there is some $d$ for which $d\nind^{\lim}_{A_\alpha}A_\kappa$ for all $\alpha<\kappa$. By symmetry of $\ind^{\lim}$, $A_\kappa\nind^{\lim}_{A_\alpha}d$ for any $\alpha<\kappa$. We will begin by extracting a countable subsequence of the $A_\alpha$'s as follows:

    Let $A_{i_0}=A_0$. Given $A_{i_j}$, we claim there is an $A_{i_{j+1}}=A_\alpha$ for some $\alpha<\kappa$ such that every formula witnessing dependence of $d$ and $A_\kappa$ over $A_{i_j}$ is realized in $A_{i_{j+1}}$. Since $A_\kappa\nind_{A_{i_j}}d$, by strong finite character there exists a formula over $\bar{A}_{i_j}d$ realized in $A_\kappa$ witnessing this dependence. For each such formula, we take a realization in $A_\kappa$, obtaining at most $\abs{A_{i_j}}\abs{T}$ many realizations. Since each $A_\alpha$ (so in particular, $A_{i_j}$) has cardinality $<\kappa$, the sequence of realizations of formulas witnessing dependence over $A_{i_j}$ cannot be cofinal in $\kappa$. Therefore we can find some $A_\alpha$ in our sequence containing all of them. Let this be our $A_{i_{j+1}}$.

    Take $A_{i_\omega}=\bigcup_{j<\omega}A_{i_j}$. Then by symmetry $A_\kappa\nind^{\lim}_{A_{i_\omega}}d$. Let $\varphi(x,y,z)$ be a formula and $a\in A_\kappa$, $b=d$, $c\in\bar{A}_{i_\omega}$ witnessing dependence as given by strong finite character. By construction there is some $j<\omega$ such that $c\in\bar{A}_{i_j}$. Then since $\varphi(x,b,c)$ is a formula over $A_{i_j}d$ realized in $A_\kappa$, choice of $A_{i_{j+1}}$ guarantees that $\varphi(x,b,c)$ is realized in  $A_{i_{j+1}}\subset A_{i_\omega}$. So in fact $A_{i_\omega}\nind^{\lim}_{A_{i_\omega}}d$, which via symmetry is a contradiction to existence.
\end{proof}

\begin{remark}\label{rem:LC-SFC}
    The only features of limit independence used in the above proof are monotonicity, symmetry, existence, and strong finite character so it is natural to conclude that these four properties together imply local character. However we do need to be a bit careful with this claim\textemdash the version of strong finite character we are using is the one from the axioms of the approximation in which we obtain a formula witness to dependence between \textit{arbitrary sets}. In general, strong finite character is typically defined for \textit{tuples} on the left in which case the above proof will not immediately follow. This distinction clarifies why $\ind^\Gamma$ in $T_\infty$ does not satisfy local character despite having these four properties\textemdash in that context, strong finite character is only defined on tuples.
\end{remark}

\begin{corollary}\label{cor:simple}
    Let $\M\models T$ be $\abs{T}^+$-saturated and approximated by a directed system of substructures of $\M$ each with a simple theory in which $\ind^f$ satisfies the independence theorem over algebraically closed sets. Then $T$ is simple.
\end{corollary}
\begin{proof}
    We have that $\ind^{\lim}$ strengthens $\ind^K$ and $T$ is $\nsop_1$. Since $\ind^{\lim}$ additionally satisfies transitivity and local character, in fact $\ind^{\lim}=\ind^K$ by \cite[Theorem 6.1]{chernikov2023transitivity}. Then since $\ind^{\lim}$ is base monotone, $\ind^K$ is base monotone in $T$ and $T$ is simple \cite[Proposition 8.8]{kaplan2020kim}. 
\end{proof}

\begin{remark}
    It is worth mentioning briefly why the previous corollary does not contradict $T_\infty$ being strictly $\nsop_1$. This is because $\ind^{\lim}$ does not have the full strength of strong finite character in the $T_\infty$ example: per \cite[Lemma 3.4.15]{harrison2013independence}, $\ind^{\lim}$ satisfies the strong finite character condition for singletons on the left, finite dimensional sets on the right, and small sets in the base. This condition is used in the arguments for several of the properties of limit independence, including base monotonicity, so those properties will not necessarily hold for $\ind^{\lim}$ outside of these specific conditions. Further, it was conjectured in \cite[Remark 3.4.17]{harrison2013independence} that the strong finite character condition would not always hold for finite tuples over an infinite-dimensional base. Using the relationship of $\ind^{\lim}$ and $\ind^{\Gamma}$ from \cite[Proposition 3.4.18]{harrison2013independence}, Remark \ref{rem:LC-SFC} suggests that the strong finite character condition does not hold in general; otherwise it and thus $\ind^{\Gamma}$ would satisfy local character, a contradiction to \cite[Theorem 12.2.2]{granger1999stability}.
\end{remark}

As demonstrated by this corollary, when we approximate strictly $\nsop_1$ theories by substructures with simple theories limit independence will not be the right notion to recover Kim-independence. However, it is still possible to approximate several well-known $\nsop_1$ examples via simple theories under this framework. In the next section we will explore these examples in detail.

For some of the examples in Section \ref{sec:examples} we will have need to remove the condition for automorphism invariance in order to get a directed system. The following is an attempt to reconcile this issue by identifying the conditions under which limit independence can be considered via a filtration of an automorphism invariant system, and does not depend on choice of filtration.

\begin{definition} Suppose $\cH$ is a (not necessarily directed) collection of substructures of $\M$ satisfying covering and automorphism invariance.
    \begin{itemize} 
        \item We say $\cH'\subseteq\cH$ is a \textbf{nice filtration of $\cH$} if $\cH'$ is a directed system satisfying conditions 1 and 3-7. 
        \item We say $\cH$ is \textbf{filtration invariant} if whenever $\cH_1,\cH_2\subseteq\cH$ are nice filtrations of $\cH$, for any $A,B,C\subseteq M$ we have $A\ind^{\lim}_CB$ in the context of $\cH_1$ if and only if $A\ind^{\lim}_CB$ in the context of $\cH_2$.
        \item We say $\cH$ is \textbf{nice} if it is filtration invariant and the class of nice filtrations of $\cH$ is nonempty and closed under automorphisms of $\M$.
    \end{itemize}
\end{definition}

\begin{theorem}\label{thm:nice-filtration}
    Let $\M$ be strongly homogeneous. If $\cH$ approximating $\M$ is nice, then we can approximate $\M$ via any nice filtration of $\cH$ (which may not satisfy condition 2) and maintain all the properties of $\ind^{\lim}$.
\end{theorem}
\begin{proof}
    The possibly missing condition of a nice filtration $\cH_1$ is automorphism invariance. The only property of $\ind^{\lim}$ that uses automorphism invariance in its proof is invariance\textemdash all other properties hold immediately from the fact that $\cH_1$ satisfies conditions 1 and 3-7.

    To show invariance, suppose $A\ind^{\lim}_CB$ and $(A,B,C)\equiv(A',B',C')$. Let $\sigma$ be an automorphism of $\M$ sending $A$ to $A'$, $B$ to $B'$, and $C$ to $C'$, and let $\cH_2=\{\sigma(\N):\N\in\cH_1\}$. By definition, $\cH_2$ is also a nice filtration of $\cH$.

    Because $A\ind^{\lim}_CB$, there is by definition for each $a\in A$ some $\N_a\in\cH_1$ containing $a$ such that for every $\N\supseteq\N_a$ in $\cH_1$, $a\ind^{\N}_{\bar{C}\cap N}\bar{B}\cap N$. By the same argument as in the original proof of invariance for $\ind^{\lim}$, we let $\sigma(\N_a)=\N_{\sigma(a)}$ witness this same condition for $\sigma(a)\in\sigma(A)=A'$, in the context of $\cH_2$. Then $A'\ind^{\lim}_{C'}B'$ in the context of $\cH_2$.

    Then since $\cH$ is filtration invariant it follows that $A'\ind^{\lim}_{C'}B'$ in the context of $\cH_1$. 
\end{proof}

\section{Examples}\label{sec:examples}

The original example of an approximable theory is the theory of an infinite dimensional vector space over an algebraically closed field with a non-degenerate symmetric or alternating bilinear form, denoted $T_\infty$. This example is detailed in \cite[Section 3.4]{harrison2013independence}, where it is shown that the approximation by substructures with finite-dimensional vector spaces satisfies all seven conditions except strong finite character of $\nind^{\lim}$, which holds only under certain conditions. At the time, this was the only known example of a theory approximated by simple substructures under Harrison-Shermoen's framework. In this section we will show two more $\nsop_1$ examples that can also be approximated in this manner. 

\subsection{Parameterized structures}

The notion of parameterized theories was first introduced by Baudisch in \cite{baudisch2002generic}, where it was considered in a single sort. An alternative two-sorted construction using simple theories obtained as \Fraisse limits satisfying the strong amalgamation property is studied in \cite{chernikov2016model}. This concept was further studied by Bossut in \cite{bossut2024some}.

We will consider a multi-sorted context, with one or more sorts for the objects and one additional sort for the parameters, and as such will restate the axioms from \cite{baudisch2002generic} in this new setting.

The construction follows the same idea as the one-sorted version. Given a theory $T$ in the language $L$ that eliminates the quantifier $\exists^\infty$ we add a new sort $P$ for parameters and define a new language $L_P$ consisting of the following, where each variable $x_i$ is an element of the object sort:
\begin{itemize}
    \item For every $n$-ary relation symbol $R(x_1,\dots,x_n)$ in $L$, we add an $(n+1)$-ary relation symbol $R_P(x_1,\dots,x_n,y)$ where $y$ is an element of the parameter sort.
    \item For every $n$-ary function symbol $f(x_1,\dots,x_n)$ in $L$, we add an $(n+1)$-ary function symbol $f_P(x_1,\dots,x_n,y)$ where $y$ is an element of the parameter sort.
    \item For every constant symbol $c$ in $L$ we add a unary function symbol $c_P(y)$ that takes an element of the parameter sort.
\end{itemize}

We will sometimes denote formulas in the parameterized language by $\varphi_P(\xx,y)$, where $\varphi(\xx)$ is a formula in $L$ and $\varphi_P$ is the formula obtained by replacing all instances of functions, relations, and constants with their $L_P$ equivalent. Given a formula $\varphi(\xx,y)$ in the parameterized language and a fixed parameter $p$, we let $(\varphi(\xx))_p$ denote the formula obtained by fixing $p$ as the input for every symbol in $L_P$. That is, $(\varphi(\xx))_p$ is equivalent to $\varphi_P(\xx,p)$. 

If $\M$ is an $L_P$-structure and $p\in\M$ an element from the parameter sort, let $\M_p$ be the $L$-structure defined on the object sort(s) of $\M$ by fixing $p$ as the parameter input to every symbol in $L_P$. That is, $\M_p\models\varphi(\aaa)$ if and only if $\M\models\varphi_P(\aaa,p)$.

Let $T_P=\{(\forall p)\; (\varphi)_p:\varphi\in T\}$. Then $\M\models T_P$ if and only if for every parameter $p\in\M$ the $L$-structure $\M_p$ is a model of $T$. Under the assumption that $T$ is model-complete and eliminates $\exists^\infty$, then the model companion $T_P^*$ exists. Let $(\varphi_i)_{i<n}$ be a collection of $L$-formulas satisfying the following conditions:
\begin{itemize}
    \item $\varphi_i$ is a conjunction of positive or negative atomic formulas in the unparameterized language.
    \item $T\models\varphi_i(\xx,\yy)\rightarrow \bigwedge_{j\neq k}x_j\neq x_k\land y_j\neq y_k\land\bigwedge x_j\neq y_k$.
\end{itemize}

For each such collection of formulas, we add the following sentences to the theory of $T_P$ to get an axiomatization of $T_P^*$:
\begin{itemize}
    \item $(\forall\yy)(\forall \text{ distinct } p_1,\dots, p_n)\;\left[\bigwedge_{i<n} (\exists^\infty\xx(\varphi_i(\xx,\yy))_{p_i} \rightarrow (\exists\xx)\bigwedge_{i<n}((\varphi_i(\xx,\yy))_{p_i})) \right]$ 
    \item $(\forall\xx)(\exists \text{ distinct }p_1,\dots,p_n)\;\left[\bigwedge_{i<n}(\varphi_i(\xx))_{p_i} \right]$, where the $\varphi_i$ each satisfy the additional condition that $\exists\xx\varphi_i(\xx)$ is consistent with $T$.
\end{itemize}

\begin{remark}\label{rem:paramsorts}
    This is a direct restatement of the one-sorted axiomatization in \cite{baudisch2002generic}, reconstructed for the two-sorted context. The proof that this axiomatizes the model companion is identical. In the case that $T$ is already a multi-sorted theory we can restrict the parameterization to only some of the object sorts and obtain the same results regarding the axiomatization of the model companion.
\end{remark}

The first set of additional sentences states that for any collection of $n$ parameters and $n$ non-algebraic formulas in the unparameterized language, we can find object variables that satisfy the $i$th formula with respect to the $i$th parameter; that is, the parameters act on the object sort independently of one another. The second set states that given any objects and $n$ realizable formulas in the unparameterized language, we can find parameters with respect to which those objects satisfy those formulas; that is, there exist parameters producing all possible behavior.

A natural consideration for the approximation of a parameterized structure is to restrict to finite sets of parameters in the substructures; that is, we let $\cH$ be the collection of substructures obtained by taking all the objects of $\M$ and only finitely many parameters. In general, however, this fails to meet the convergence of truth value condition.

\begin{lemma}
    Suppose $T$ is a theory with $D$-rank $\geq 2$. Then $T_P^*$ cannot be approximated via restriction to finitely many parameters.
\end{lemma}
\begin{proof}
    The assumption that the $D$-rank of $T$ is $\geq 2$ implies that there is an $L$-formula $\varphi(x;y)$ and an indiscernible sequence $(b_i)_{i<\omega}$ such that $\{\varphi(x;b_i):i<\omega\}$ is $k$-inconsistent for some finite $k$ and the set defined by $\varphi(x;b_i)$ is infinite. 
    
    Consider the formula $\psi$ given by $(\exists x)(\forall p)(\varphi(x;b_0))_p$. From the axioms for parameterized theories, the fact that $\varphi(x;b_0)$ is nonalgebraic guarantees that for any finite set of parameters $p_1,\dots,p_n$ the set $\{(\varphi(x;b_0))_{p_i}:i<n\}$ is consistent. Then for every $\N\in\cH$, since there are only finitely many parameters in $\N$, the set $\{(\varphi(x;b_0)_p:p\in N\}$ is consistent, so $\N\models\psi$.

    Because $\{\varphi(x;b_i):i<\omega\}$ is $k$-inconsistent it follows that $\exists x\neg\varphi(x,b_0)$ is consistent with $T$. Again from the axioms for parameterized theories, we have in $\M$ that for every $a$ there exists a $p$ such that $\neg(\varphi(a,b_0))_p$ holds. Then in $\M$ the set $\{(\varphi(x;b_0))_p:p\in M\}$ is inconsistent, so $\M\not\models\psi$. Thus $\M\models\neg\psi$. However, $\N\models\psi$ for every structure $\N\in\cH$ so $\psi$ cannot be eventually false in $\cH$ and convergence of truth value fails.
\end{proof}

We will now give an example of a parameterized theory that fits the framework.

\subsubsection{$T^*_{\feq}$: the generic theory of a parameterized equivalence relation.} $T^*_{\feq}$ is the two-sorted theory obtained by parameterizing the theory of an equivalence relation with infinitely many classes, all of which are infinite. The language is given by the single ternary relation $E_x(y,z)$, where $x$ is from the parameter sort and $y$ and $z$ from the object sort. The theory states that $E_p(y,z)$, for any fixed parameter $p$, defines an equivalence relation on the object sort with infinitely many classes, all of which are infinite. Furthermore, the genericity condition guarantees that for any two distinct elements of the object sort there is a parameter that distinguishes them. Note that this is a special case of the parameterized theories described above.

We write $\M=(A,B)$ for a model of $T^*_{\feq}$, where $A$ is the set of objects and $B$ the parameters. Given $p\in B$, we write $A_p$ for the structure in the unparameterized language with domain $A$ with the equivalence relation $E(x,y)$ interpreted by $E_p(x,y)$.

We will approximate $\M\models T_{\feq}^*$ by substructures where we restrict the number of classes in the equivalence relation with respect to each parameter. 

The age of $\M\models T^*_{\feq}$ is the \Fraisse class $K_{\feq}=\{M=(A,B):\forall b\in B\; A_b\text{ is an equivalence relation}\}$. We use the filtration of $K_{\feq}$ by the \Fraisse classes $K_n$ from \cite[Section 4.2]{kruckman2019disjoint}, defined as follows: for each $n<\omega$ $K_n\subset K_{\feq}$ is the set \[K_n=\{M=(A,B):\forall b\in B\; A_b\text{ is an equivalence relation with at most $n$ classes}\}.\]

Let $\M_n$ denote the \Fraisse limit of $K_n$ and let $T_n=\Th(\M_n)$. Then $T_n$ is the theory of a parameterized equivalence relation with $n$ infinite classes with respect to each parameter. 

To build our approximation to be homogeneous, we need slightly more structure. We extend the language by countably many unary predicates to get $\lang_Q=\lang_{\feq}\cup\{Q_i,R_i\}_{i<\omega}$. We say $(Q_n,R_n)$ \textit{imposes $K_n$-structure} on $M=(A,B)$ if $Q_n^M\subseteq A, R_n^M\subseteq B$, and for every $b\in R_n(B)$, $Q_n(A_b)$ represents at most $n$ classes of the equivalence relation. Now let
\[K_Q=\{M:M\restriction L_{\feq}\in K_{\feq}\text{ and } \forall i<\omega, (Q_i^M,R_i^M)\subseteq(Q_{i+1}^M,R_{i+1}^M)\text{ and } (Q_i,R_i)\text{ imposes $K_i$-structure on $M$}\}.\]

\begin{lemma}
    $K_Q$ is a \Fraisse class.
\end{lemma}
\begin{proof}
    HP is clear and since the empty structure is a model in $K_Q$, JEP follows from SAP. So we just need to show SAP.

    Suppose $(A,D),(B,E),$ and $(C,F)$ are in $K_Q$ and we have embeddings:
    \[\begin{tikzpicture}
        \node (a) at (0,0) {$(A,D)$};
        \node (b) at (2,1) {$(B,E)$};
        \node (c) at (2,-1) {$(C,F)$};

        \draw[->] (a) edge["$g$"] (b);
        \draw[->] (a) edge["$h$"] (c);
    \end{tikzpicture}\]
    We can assume $E\cap F=D$ and the embeddings are just inclusions on the parameter sort by moving $E$ and $F$ over $D$.

    Let $G=A\cup (B\setminus g(A))\cup (C\setminus h(A))$ and $H=D\cup (E\setminus g(D))\cup (F\setminus h(D))$, with $(B,E)$ and $(C,F)$ embedding in the obvious way. We build the equivalence relation on $(G,H)$ following the same method as in the proof of \cite[Theorem 4.5]{kruckman2016infinitary}. Given $p\in H$, if $p\in D$ then it already defines equivalence relations on $B$ and $C$. Enumerate the $E_p$-classes in $A$ by $1,\dots,\ell$ and the additional unnumbered $E_p$-classes in $B$ and $C$ by $\ell+1,\dots,m_B$ and $\ell+1,\dots,m_C$ respectively. Then on $G$ we define $E_p$ to have $\max(m_B,m_C)$ classes by merging those assigned the same number. If $p\notin D$, then without loss of generality say $p\in E$ and assign all elements of $C\setminus g(A)$ to a single existing $E_p$-class. We have that the following diagram commutes in $L_{\feq}$:
    \[\begin{tikzpicture}
        \node (a) at (0,0) {$(A,D)$};
        \node (b) at (2,1) {$(B,E)$};
        \node (c) at (2,-1) {$(C,F)$};
        \node (g) at (4,0) {$(G,H)$};

        \draw[->] (a) edge["$g$"] (b);
        \draw[->] (a) edge["$h$"] (c);
        \draw[->] (b) edge["$g'$"] (g);
        \draw[->] (c) edge["$h'$"] (g);
    \end{tikzpicture}\]

    For each $n<\omega$, let $Q_n(G)=g'(Q_n(B))\cup h'(Q_n(C))$ and $R_n(H)=g'(R_n(E))\cup h'(R_n(F))$. We claim that the diagram then commutes in $L_Q$. It remains to show that each $(Q_i,R_i)$ imposes $K_i$-structure on $(G,H)$.

    Fix $n<\omega$ and let $p\in R_n(H)$. If $p\in R_n(D)$, then $p$ defines equivalence relations on $Q_n(B)$ and $Q_n(C)$ with at most $n$ classes. Due to the construction of the equivalence relation on $G$, the number of classes in the amalgam $Q_n(G)$ is $\max(m_B,m_C)$ where $m_B$ and $m_C$ are the number of $E_p$-classes in $Q_n(B)$ and $Q_n(C)$ respectively. Both $m_B$ and $m_C$ are necessarily $\leq n$, so $Q_n(G)$ has at most $n$ classes.

    If instead $p\notin R_n(D)$, then without loss of generality suppose $p\in R_n(E)$. Again by construction, the number of classes in the amalgam $Q_n(G)$ will equal the number of classes in $Q_n(B)$, as all new objects from $Q_n(C)$ are placed into one of the existing classes already defined on $Q_n(B)$. Therefore we cannot increase the number of classes, so since $Q_n(B)$ had at most $n$ classes, the same holds for $Q_n(G)$.

    Since $n$ was arbitrary, this holds for all $n<\omega$, and we conclude SAP for $K_Q$.
\end{proof}

Let $\M^Q$ denote the \Fraisse limit of $K_Q$. Note that the $\lang_{\feq}$-structure given by $(Q_n(\M^Q),R_n(\M^Q))$ is a model of $T_n$. 

\begin{lemma}
    $\M^Q\restriction L_{\feq}\models T^*_{\feq}$.
\end{lemma}
\begin{proof}
    From \cite[Lemma 2.12]{ramsey2020invariants}, we need to show that the following conditions are satisfied:
    \begin{enumerate}
        \item $A\in K_{\feq}$ if and only if there is a $D\in K_Q$ such that $A$ is an $L_{\feq}$ substructure of $D\restriction L_{\feq}$.
        \item If $A,B\in K_{\feq}$, $\pi:A\to B$ is an $L_{\feq}$ embedding, and $C\in K_Q$ with $C=\langle A\rangle_{L_Q}^C$, then there is a $D\in K_Q$ such that $B$ is an $L_{\feq}$ substructure of $D\restriction L_{\feq}$ and an $L_Q$ embedding $\tilde{\pi}: C\to D$ extending $\pi$.
    \end{enumerate}
    Condition 1 follows immediately from the definition of $K_Q$. 

    Let $A,B,C,\pi$ as in condition 2. Here $C$ and $A$ have the same underlying set, so we define $D$ to be the structure with underlying set $B$ and for each $n<\omega$ let $Q_n(D)=\pi(Q_n(C))$ and $R_n(D)=\pi(R_n(C))$. Let $\tilde{\pi}=\pi$. Clearly $\tilde{\pi}$ is an $L_Q$ embedding, so we just need to show that $D\in K_Q$. Let $p\in R_n(D)$ for some $n<\omega$ and consider $Q_n(D)_p$. We know that $\tilde{\pi}^{-1}(p)\in R_n(C)$ and $C\in K_Q$, so the preimage $\tilde{\pi}^{-1}(Q_n(D))=Q_n(C)$ has $\leq n$ classes with respect to $\tilde{\pi}^{-1}(p)$. Then since the restriction of $\tilde{\pi}$ to $Q_n$ and $R_n$ is onto, $Q_n(D)$ must also have $\leq n$ classes with respect to $p$. By definition $D\restriction L_{\feq} = B\in K_{\feq}$, so $D\in K_Q$ and we are done.
\end{proof}

Now we construct our approximation as follows. For each $n<\omega$ let $\M_n\subseteq\M^Q$ be the $L_{\feq}$-substructure identified by the predicates $Q_n,R_n$. The reduct of $\M_n$ to $L_{\feq}$ is a model of $T_n$, and by the nesting condition on the predicates in $K_Q$, we can write them as a chain of substructures of $\M^Q\restriction L_{\feq}$: \[\M_i\subseteq \M_{i+1}\subseteq\dots\subseteq \M^Q\restriction L_{\feq}\] Now let $\M=\bigcup_{n<\omega}\M_n$. Since $(K_n)_{n<\omega}$ is a filtration of $K_{\feq}$ and the chain of $\M_n$'s is an increasing union of the corresponding \Fraisse limits, $\M\models T^*_{\feq}$.

For our approximation, we let $\cH$ be the collection of every substructure of $\M$ that is isomorphic to some $\M_n$ for $n<\omega$. Note that we may have two structures $\N$ and $\N'$ that are each isomorphic to some $\M_n$ (hence both in $\cH$) but with the objects of $\N'$ having infinitely many classes with respect to a parameter from $\N$. Thus $\N\cup\N'$ cannot be contained in any element of $\cH$, so $\cH$ is not a directed system.

To handle this, we'll show that $\cH$ is nice and therefore that we can reduce to the filtration of $\cH$ to the directed system given by the constructed chain of $\M_n$'s.

\begin{remark}\label{rem:trivfork}
    By Theorem 3.14 and the proof of Theorem 4.5 in \cite{kruckman2016infinitary}, $\Th(\M_n)$ has trivial forking and trivial $\acl$ for every $n<\omega$. It follows that the same is true for every element of $\cH$. In particular, this implies that the theory of each $\M_n$ is simple.
\end{remark}

\begin{lemma}\label{lem:feq-nice}
    $\cH$ is nice.
\end{lemma}
\begin{proof}
    Suppose $\cH_1,\cH_2\subseteq\cH$ are nice filtrations and let $A\ind^{\lim}_CB$ in the context of $\cH_1$. Suppose $A\nind^{\lim}_CB$ in the context of $\cH_2$. Then there is some $a\in A$ such that for every $\N_a\in\cH_2$ there is an $\N\supseteq\N_a$ in $\cH_2$ such that $a\nind^{\N}_{\bar{C}\cap N}\bar{B}\cap N$. However, since forking is trivial in the elements of $\cH$, it follows that $a\in B\cap N$. 

    Let $\N_a\in\cH_1$ be as in the definition of limit independence. Since $a\in\N_a$, it follows that $a\in B\cap N_a$ and therefore $a\nind^{\N_a}_{\bar{C}\cap N_a}\bar{B}\cap N_a$, a contradiction. So $\cH$ is filtration invariant.

    Now suppose $\cH_1$ is a nice filtration of $\cH$ and $\sigma$ an automorphism of $\M$. Let $\cH_2=\{\sigma(\N):\N\in\cH_1\}\subseteq\cH$. We claim $\cH_2$ is a nice filtration of $\cH$.

    That $\cH_2$ is directed and satisfies covering follows immediately from the definition. Convergence of truth value and homogeneity follow from preservation of truth value under automorphism. Automorphism invariance of $\ind^{\N}$ for any $\N\in\cH$ gives stabilization of non-forking. Together these imply strong finite character.

    Stabilization of algebraic closure follows from trivial acl for every $\N\in\cH$.
\end{proof}
\begin{remark}\label{rem:nice-acl}
    Note that for every condition other than stabilization of algebraic closure, the proof that the condition is preserved under automorphisms does not depend on any specifics of the theory. Stabilization of acl will not necessarily be maintained in general, as the automorphism taking elements of $\cH_1$ to elements of $\cH_2$ is not guaranteed to fix every set $A$ pointwise.
\end{remark}

Now let $\cH'\subset\cH$ be the collection of substructures $\M_n$. We will show that $\cH'$ is a nice filtration of $\cH$, and therefore that we can approximate $\M$ via $\cH'$. The argument largely follows the same ideas presented in \cite[Section 4.2]{kruckman2019disjoint}.

\begin{lemma}\label{lem:feq-ctv}
    (Convergence of truth value). For any $\aaa\in M$ and formula $\varphi(\xx)\in\lang_{\feq}$, if $\M\models\varphi(\aaa)$ then $\varphi(\aaa)$ is eventually true in $\cH'$.
\end{lemma}

\begin{proof}
    We say that a formula $\varphi$ has uniform quantifier elimination for large enough $n$ if there is some fixed $N<\omega$ such that if $\psi$ is the quantifier free formula given by quantifier elimination in $\M$ such that $\M\models(\forall\xx)[\varphi(\xx)\leftrightarrow\psi(\xx)]$, then $\M_n\models(\forall\xx)[\varphi(\xx)\leftrightarrow\psi(\xx)]$ for all $n\geq N$. To prove convergence of truth value, we will first show by induction on formulas that every formula in $\lang_{\feq}$ has uniform quantifier elimination for large enough $n$.

    If $\varphi$ is quantifier-free the conclusion is immediate. Suppose $\varphi'$ is a (not necessarily quantifier free) formula satisfying uniform quantifier elimination for large enough $n$.

    Let $\varphi(\xx)=\exists\yy\varphi'(\xx,\yy)$. By the inductive hypothesis, for $n\geq N$ we can replace $\varphi'$ by its quantifier-free equivalent and consider $\varphi$ as a formula with a single existential quantifier. Consider an arbitrary $\aaa$. If $\M_n\models\varphi(\aaa)$, then since $\varphi$ is existential it follows $\M\models\varphi(\aaa)$. By quantifier elimination in $\M$, $\M\models\psi(\aaa)$, and since $\psi$ is quantifier-free it follows that $\M_n\models\psi(\aaa)$ in any $\M_n$ containing $\aaa$, so the forwards implication holds. Now suppose $\M_n\models\psi(\aaa)$. It follows that $\M\models\psi(\aaa)$ and by quantifier elimination that $\M\models\varphi(\aaa)$. We claim that $\M_n\models\varphi(\aaa)$ for large enough $n$.

    Since $\M\models\varphi(\aaa)$ there is some $\bb\in\M$ such that $\M\models\varphi'(\aaa,\bb)$; however $\bb$ is not necessarily contained in $\M_n$. Let $n\geq\max(\abs{\aaa\bb},N)$. Since $\aaa\in\M_n$ we have $\vect{\aaa}\in K_n$, so to show $\vect{\aaa\bb}\in K_n$, we need only show that the elements of $\bb$ cannot increase the number of equivalence classes in any parameter to be greater than $n$. 

    A parameter in $\bb$ can partition the set into at most $\abs{\aaa\bb}-1$ classes, in the case where all other elements are objects and the parameter places each object into a distinct class. An object in $\bb$ can increase the number of classes represented for a given parameter by at most 1, so collectively if $\bb$ contains only objects the number of new classes not represented by elements of $\aaa$ is at most $\abs{\bb}$. Since $\vect{\aaa}$ can represent at most $\abs{\aaa}$ classes (in the case where each element of $\aaa$ is an object of distinct class) the total number of classes in $\abs{\aaa\bb}$ with respect to any parameter cannot exceed $\abs{a}+\abs{b}=\abs{\aaa\bb}\leq n$. Therefore $\vect{\aaa\bb}$ represents $\leq n$ classes with respect to each parameter and $\vect{\aaa\bb}\in K_n$, so $\vect{\aaa\bb}$ embeds into $\M_n$ over $\vect{\aaa}$.

    Now, using this embedding we can find a $\bb'\in\M_n$ such that $\M_n\models\varphi'(\aaa,\bb')$. Then $\bb'$ is a witness to the existential formula $\varphi(\aaa)$, so $\M_n\models\varphi(\aaa)$ and the backwards implication holds.

    The universal case is simply the negation of the existential case and follows the same argument to show that the bi-implication holds.

    Therefore every formula has uniform quantifier elimination for large enough $n$.

    Let $\aaa\in\M$ and $\varphi(\xx)\in\lang_{\feq}$ such that $\M\models\varphi(\aaa)$. By quantifier elimination in $\M$ we can replace $\varphi$ by an equivalent quantifier-free formula $\psi$. Then since $\psi$ is quantifier-free, $\M_n\models\psi(\aaa)$ for any $\M_n$ containing $\aaa$. Since $\varphi$ has uniform quantifier elimination for large enough $n$, it follows that for large enough $n$, $\M_n\models\varphi(\aaa)$ and $\varphi(\aaa)$ is eventually true.
\end{proof}

\begin{theorem}\label{thm:feq-approx}
    $\cH'$ is a nice filtration of $\cH$ and thus approximates $\M$. 
\end{theorem}

\begin{proof}
\begin{enumerate}
    \item \textbf{Covering}: By construction.
    \item \textbf{Convergence of truth value}: This is Lemma \ref{lem:feq-ctv}.
    \item \textbf{Stabilization of non-forking}: Follows from trivial forking in the approximating theories (Remark \ref{rem:trivfork}).
    \item \textbf{Strong finite character of limit independence}: Follows from trivial forking in the approximating theories (Remark \ref{rem:trivfork}).
    \item \textbf{Stabilization of algebraic closure}: Follows from trivial $\acl$ of the approximating theories (Remark \ref{rem:trivfork}).
    \item \textbf{Homogeneity}: It follows from the construction of $\cH'$ that given $A,B$ finite subsets of $\N=\M_n$ and $f:A\cong B$ an isomorphism we can find $\sigma\in\aut(M)$ extending $f$ such that $\sigma\restriction M_n\in\aut(M_n)$.
\end{enumerate}

Note that $\cH$ clearly satisfies covering and automorphism invariance. By Lemma \ref{lem:feq-nice}, $\cH$ is nice, so by Theorem \ref{thm:nice-filtration} any nice filtration of $\cH$ will approximate $\M$ maintaining all the properties of $\ind^{\lim}$. In particular, since $\cH'$ is a nice filtration, $\cH'$ approximates $\M$. 
\end{proof}

\begin{remark}
    Theorem \ref{thm:feq-approx} does not imply that the theory $T^*_{\feq}$ is simple because we are working in a countable model and not the monster model so Corollary \ref{cor:simple} doesn't apply. Furthermore, $\ind^f$ does not satisfy the independence theorem over algebraically closed sets in the approximating theories, so we cannot conclude the independence theorem for $\ind^{\lim}$. 
\end{remark}

\subsection{$\omega$-free PAC fields} Another well-studied $\nsop_1$ theory is that of $\omega$-free PAC fields. We say that a field $K$ is \emph{pseudo-algebraically closed} (PAC) if every absolutely irreducible variety defined over $K$ has a $K$-rational point, and it is \emph{$\omega$-free} if there is an elementary substructure whose absolute Galois group is isomorphic to $\hat{F}_\omega$, the free profinite group on countably many generators. We say $K$ is \emph{$e$-free} if its absolute Galois group is isomorphic to $\hat{F}_e$, the free profinite group on $e$ generators.

Throughout this section we will use field theoretic knowledge and terminology from \cite{friedjarden}. We begin with some relevant definitions:

\begin{definition}

\begin{enumerate}
    \item A field extension $F/K$ is \textit{regular} if $F/K$ is separable and $K$ is algebraically closed in $F$; or equivalently, if $F$ is linearly disjoint from $\tilde{K}$ over $K$.
    \item If $K$ is a field and $K_0\subseteq K$ the prime subfield of $K$, then the \textit{field of absolute numbers of $K$} is given by $K_0^{\alg}\cap K$. 
    \item A field is called \textit{bounded} if for any integer $n$ it has only finitely many Galois extensions of degree $n$.
\end{enumerate}
\end{definition}

We will additionally assume that the field being approximated is perfect, and that the absolute numbers have finitely generated Galois group (in a topological sense). In other words, we assume the absolute numbers are a bounded field.

We will build our approximation of $K$ as a collection of countable $e$-free PAC fields that are each regular in $K$. To do so we first need the following lemmas.

\begin{lemma}
    (PAC Embedding Lemma, \cite[23.2.2]{friedjarden}) Let $E/L$ and $F/M$ be separable field extensions satisfying: $E$ is countable and $F$ is PAC and $\aleph_1$-saturated. In addition, suppose that there are an isomorphism $\Phi_0: L^{\sep}\to M^{\sep}$ with $\Phi_0(L)=M$ and a commutative diagram
    \begin{center}
        \begin{tikzpicture}
            \node (E) at (0,2) {$G(E)$};
            \node (F) at (2,2) {$G(F)$};
            \node (L) at (0,0) {$G(L)$};
            \node (M) at (2,0) {$G(M)$};
            \draw[->] (F) edge["$\varphi$"] (E);
            \draw[->] (M) edge["$\varphi_0$"] (L);
            \draw[->] (E) edge["$\res$"] (L);
            \draw[->] (F) edge["$\res$"] (M);
        \end{tikzpicture}
    \end{center}
    where $\varphi_0$ is the isomorphism induced by $\Phi_0$ and $\varphi$ is a homomorphism. Then, there exists an extension of $\Phi_0$ to an embedding $\Phi: E^{\sep}\to F^{\sep}$ which induces $\varphi$ with $F/\Phi(E)$ separable.
\end{lemma}

\begin{lemma}\label{lem:pacbase}
    Let $K$ be a perfect $\omega$-free $\aleph_1$-saturated PAC field with bounded absolute numbers. Then there exists an $e$-free PAC field $F$ with the same absolute numbers over which $K$ is regular.
\end{lemma}
\begin{proof}
    Let $E:=\acl^K(\varnothing)$ have its absolute Galois group finitely generated by $e$ generators (topologically). Then by \cite[Proposition 38]{cherlin1980elementary} there exists an $e$-free PAC field $F$ with $\acl^F(\varnothing)=E$. There exists a surjection $\pi:G(K)\to G(F)$ and both $G(K)$ and $G(F)$ surject onto $G(E)$ via the restriction map. We have the following diagram:
    \begin{center}
        \begin{tikzpicture}
            \node (F) at (0,2) {$G(F)$};
            \node (K) at (2,2) {$G(K)$};
            \node (L) at (0,0) {$G(E)$};
            \node (M) at (2,0) {$G(E)$};
            \draw[->] (K) edge["$\pi$"] (F);
            \draw[->] (M) edge["id"] (L);
            \draw[->] (F) edge["res"] (L);
            \draw[->] (K) edge["res"] (M);
        \end{tikzpicture}
    \end{center}
    By the PAC embedding lemma we can find a copy $F'$ of $F$ inside $K$ with $\res_{K/F'}$ inducing $\pi$, so $K/F'$ is regular. 
\end{proof}

This gives the base case for our approximation. The next theorem shows how we can inductively build an $e+1$-free PAC field that is regular in $K$ from an $e$-free PAC field regular in $K$. To prove it we will need a few more definitions:

\begin{definition} 
    Let $G$ be a profinite group. 
    \begin{enumerate}
    \item We say $G$ is \textit{projective} if whenever $\alpha:G\to A$ and $\beta:B\to A$ are epimorphisms for $A,B$ finite groups, there is a homomorphism $\gamma:G\to B$ so that the following diagram commutes:
    \begin{center}
    \begin{tikzpicture}
        \node (g) at (0,1.5) {$G$};
        \node (a) at (0,0) {$A$};
        \node (b) at (2,0) {$B$};
        \draw[->>] (g) edge["$\alpha$"] (a);
        \draw[->>] (b) edge["$\beta$"] (a);
        \draw[->] (g) edge["$\gamma$"] (b);
    \end{tikzpicture}
    \end{center} 
    \item We say $G$ has the \textit{embedding property} if whenever $\alpha:G\to A$ and $\beta:B\to A$ are epimorphisms and $B\in\Image(G)$ there is an epimorphism $\gamma:G\to B$ so that the following diagram commutes.
    \begin{center}
    \begin{tikzpicture}
        \node (g) at (0,1.5) {$G$};
        \node (a) at (0,0) {$A$};
        \node (b) at (2,0) {$B$};
        \draw[->>] (g) edge["$\alpha$"] (a);
        \draw[->>] (b) edge["$\beta$"] (a);
        \draw[->>] (g) edge["$\gamma$"] (b);
    \end{tikzpicture}
    \end{center}
    \item If $G$ is both projective and has the embedding property, then $G$ is \textit{superprojective}.
    \end{enumerate}
\end{definition}

\begin{theorem}\label{thm:pacapprox}
    Let $K$ be a perfect $\omega$-free $\aleph_1$-saturated PAC field with bounded absolute numbers. Let $F$ be a countable $e$-free PAC field with $K/F$ regular. Then there exists an $e+1$-free PAC field $F'\supseteq F$ such that $K/F'$ is regular.
\end{theorem}
\begin{proof}
    Let $K_0\preceq K$ be a countable elementary submodel containing $F$ with $K_0/F$ regular.  
    We have the following commutative diagram:
    \begin{center}
    \begin{tikzpicture}
        \node (k) at (0,1.5) {$G(K_0)$};
        \node (f) at (0,0) {$G(F)$};
        \node (l) at (2,0) {$\hat{F}_{e+1}$};
        \draw[->>] (k) edge["res"] (f);
        \draw[->>] (l) edge["$\pi$"] (f);
        \draw[->>] (k) edge["$\varphi$"] (l);
    \end{tikzpicture}
    \end{center} where res is surjective due to regularity, $\pi$ can be taken as the surjection mapping the first $e$ generators of $\hat{F}_{e+1}$ to the generators of $G(F)$ and the last to the identity, and $\varphi$ is the surjection given from superprojectivity of free profinite groups.

    Given $\varphi$, we also have the following diagram:
    \begin{center}
    \begin{tikzpicture}
        \node (k) at (0,1.5) {$G(K_0)$};
        \node (f) at (0,0) {$\hat{F}_{e+1}$};
        \node (l) at (2,0) {$\hat{F}_{e+1}$};
        \draw[->>] (k) edge["$\varphi$"] (f);
        \draw[->>] (l) edge["id"] (f);
        \draw[->] (l) edge[right,"$\psi$"] (k);
    \end{tikzpicture}
    \end{center} where the existence of $\psi$ comes from projectivity of $\hat{F}_{e+1}$. Further, since the diagram commutes we have that $\psi$ is an injection, so the image $\psi(\hat{F}_{e+1})$ is a copy of $\hat{F_{e+1}}$ contained within $G(K_0)$. Let $F'$ denote the fixed field $(K_0^{\alg})^{\psi(\hat{F}_{e+1})}$. From the diagrams we have $\res\circ\psi$ is surjective since $\pi$ is surjective, so the restriction map applied to $G(F')=\psi(\hat{F}_{e+1})$ is surjective onto $G(F)$ and therefore $F'/F$ is regular. Since algebraic extensions of PAC fields are PAC, $F'$ is PAC \cite[Corollary 12.2.5]{friedjarden}.

    Finally we apply the PAC embedding lemma to the following diagram to get a copy of $F'$ inside of $K$:
    \begin{center}
        \begin{tikzpicture}
            \node (E) at (0,2) {$G(F')$};
            \node (F) at (2,2) {$G(K)$};
            \node (L) at (0,0) {$G(K_0)$};
            \node (M) at (2,0) {$G(K_0)$};
            \draw[->] (F) edge["$\varphi$"] (E);
            \draw[->] (M) edge["id"] (L);
            \draw[->] (E) edge["res"] (L);
            \draw[->] (F) edge["res"] (M);
        \end{tikzpicture}
    \end{center} 
    
\end{proof}

To build $\cH$, we start with the $e$-free PAC field constructed from $\acl(\varnothing)$ via Lemma \ref{lem:pacbase}. Then we apply Theorem \ref{thm:pacapprox} repeatedly to build a chain of $e$-free PAC fields,  with $e$ increasing, over each of which $K$ is regular. 

Because of the necessary assumptions on $K$ to construct this chain, in order to get the covering condition we will not be able to approximate every example of an $\omega$-free PAC field; however, for every theory $T$ of a perfect $\omega$-free PAC field with bounded absolute numbers we can find a $K\models T$ which is approximable. That is, we will not consider the chain necessarily as an approximation of the field $K$ itself, but of the field $K'=\bigcup\cH$. Since the theory of a perfect $\omega$-free PAC field is determined by fixing the characteristic and the absolute numbers \cite[Theorem 4.2]{jarden1976elementary}, $K'\models\Th(K)$. From now on we will write $K$ to mean this representative of its theory and the field being approximated.

Let $\cH$ be the collection of all substructures of $K$ that are isomorphic to some $F$ in this chain. We will show that $\cH$ is nice, but to do so we first need a bit more information.

\begin{fact}\label{fact:pac-fork}
    Forking in elements of $\cH$ is given by $a\ind^f_Cb\iff\acl(aC)\ind_{\acl(C)}^{\ACF}\acl(bC)$.
\end{fact}
It follows from \cite{hrushovski1991pseudo} by the Kim-Pillay theorem that this relation is forking independence in a bounded PAC field. This is treated in detail in \cite{chatzidakis1998generic}.

\begin{lemma}\label{lem:aclpacsubsets}
    For any $X\subseteq K_n\in\cH$, $\acl_{K_n}(X)=\acl_K(X)$.
\end{lemma}
\begin{proof}
    Since $K_n$ is a perfect PAC field, $\acl_{K_n}(X)=X^{\alg}\cap K_n$ \cite[Corollary 4.5]{chatzidakis1998generic}. From $K/K_n$ regular, we have that $K$ is linearly disjoint from $K_n^{\alg}$ over $K_n$; in particular, $K_n^{\alg}\cap K=K_n$. Replacing $K_n$, we then have $\acl_{K_n}=X^{\alg}\cap K_n^{\alg}\cap K$. Then since $X\subseteq K_n$ we know $X^{\alg}\subseteq K_n^{\alg}$, so $X^{\alg}\cap K_n^{\alg}=X^{\alg}$. Then since $K$ is PAC we have $\acl_{K_n}(X)=X^{\alg}\cap K=\acl_K(X)$.
\end{proof}
\begin{corollary}\label{cor:pacstabacl}
    (Stabilization of algebraic closure) If $\bar{A}=A$ and $\cH'\subseteq\cH$ is a filtration of $\cH$ to a directed system then there is $\N_{\alg}\in\cH'$ such that for all $\N\supseteq\N_{alg}$, $A\cap N=\acl^\N(A\cap N)$.
\end{corollary}

Note that stabilization of algebraic closure does not depend on the choice of filtration; in fact, any filtration will satisfy it.

\begin{lemma}\label{lem:pac-nice}
    $\cH$ is nice.
\end{lemma}
\begin{proof}
        Suppose $\cH_1,\cH_2\subseteq\cH$ are nice filtrations and let $A\ind^{\lim}_CB$ in the context of $\cH_1$. Suppose $A\nind^{\lim}_CB$ in the context of $\cH_2$. Then there is some $a\in A$ such that for every $\N_a\in\cH_2$ there is an $\N\supseteq\N_a$ in $\cH_2$ such that $a\nind_{\bar{C}\cap N}^{\N}\bar{B}\cap N$. Fix an arbitrary $\N_a$ and let $\varphi(\xx,\yy)$ be a formula witnessing forking in the corresponding $\N$. By Fact \ref{fact:pac-fork}, $\varphi$ is a quantifier-free formula in the language of rings defined over $C\cap N$ realized by $a,b$ for some $b\in B$, such that the dimension of $\varphi(\xx,\yy)$ is a definable property of the parameters $\yy$ \cite[Corollary 1.10]{marker2005model}.

    Let $\N_1\in\cH_1$ be the structure given by the definition of limit independence such that for all $\N\supseteq\N_1$ in $\cH$, $a\ind_{\bar{C}\cap N}^N\bar{B}\cap N$. Let $\N'\in\cH$ contain $\N_1\cup\{b\}$. Then $\varphi(\xx,\yy)$ witnesses forking in $\N$, which is a contradiction to choice of $\N_1$. 

    By Remark \ref{rem:nice-acl}, to show automorphisms of $K$ preserve nice filtrations we need only show that stabilization of $\acl$ is preserved. This is immediate from Corollary \ref{cor:pacstabacl}.    
\end{proof}

As in the previous example, we will fix a filtration of $\cH'$ to be the chain of $e$-free PAC fields constructed using Theorem \ref{thm:pacapprox}. We then have the following corollaries: 
\begin{corollary}\label{cor:pacsnf}
    (Stabilization of non-forking) For every formula or finite partial type $\pi(x,y)$, tuple $b$, and set $C$ we have one of the following
    \begin{itemize}
        \item $\pi(x,b)$ eventually forks over $C$ in $\cH'$
        \item $\pi(x,b)$ eventually does not fork over $C$ in $\cH'$
    \end{itemize} 
\end{corollary}

\begin{corollary}\label{cor:pacsfc}
    (Strong finite character of $\nind^{\lim}$) If $A\nind^{\lim}_CB$, then there are $\varphi(x,y,z)$ without parameters, $a\in A$, $b\in\bar{B}$, and $c\in\bar{C}$ such that $\M\models\varphi(a,b,c)$ and $\varphi(x,b,c)$ eventually forks over $\bar{C}$ in $\cH'$.
\end{corollary}

\begin{definition}\cite[Definition 2.1]{chatzidakis2023measures}
    Let $E$ be a field. We say a \textit{test sentence over $E$} is a Boolean combination of $L(E)$ sentences of the form $\exists y\;f(y)=0$ where $f(y)\in E[y]$ (for a single variable $y$). We say a \textit{test formula over $E$} is a Boolean combination of $L(E)$ formulas of the form $\exists y\;f(x,y)=0$ where $f(x,y)\in E[x,y]$ and $y$ is again a single variable.
\end{definition}

\begin{lemma}\label{lem:cor2.4}
    \cite[Corollary 2.4]{chatzidakis2023measures} Let $E$ be a subfield of an $e$-free PAC field $K$ and $a,b$ be tuples in $K$. The following are equivalent:
    \begin{enumerate}
        \item $\tp(a/E)=\tp(b/E)$
        \item There is an $E$-isomorphism $E(A)^{\sep}\cap K\to E(b)^{\sep}\cap K$ taking $a$ to $b$
        \item $a$ and $b$ satisfy the same test-formulas over $E$
        \item For every finite Galois extension $L$ of $E(a)$ there is a field embedding $\varphi:L\to E(b)^{\sep}$ such that $\varphi(L\cap K)=\varphi(L)\cap K$. 
    \end{enumerate}
\end{lemma}

\begin{lemma}\label{lem:qetest}
    For all $e\in\omega\cup\{\omega\}$, let $E$ be a subfield of an $e$-free PAC field $K$ and $\varphi(\xx)$ a formula in $L(E)$. Then there exists a test formula $\psi(\xx)$ over $E$ such that $K\models(\forall\xx)\;(\varphi(\xx)\leftrightarrow\psi(\xx))$. 
\end{lemma}
\begin{proof}
    The finite case is a corollary of Lemma \ref{lem:cor2.4} as the type of any $a\in K$ is entirely determined by its truth value on test formulas over $E$. The $\omega$-case follows from \cite[Section 4.1]{chatzidakis2023measures}.
\end{proof}

\begin{lemma}
    Let $F\prec K$ be countable, $e$-free, and PAC with $K/F$ regular, $\varphi(\xx)$ a test formula over $F$, and $\aaa\in K\setminus F$ such that $K\models\varphi(\aaa)$. Then $\varphi(\aaa)$ is eventually true in $\cH$.
\end{lemma}
\begin{proof}
    Since $\varphi$ is a test-formula over $F$ there exists a finite Galois extension $M/F(\aaa)$ such that the question of whether a field $K'\models\varphi(\aaa)$ depends only on $K'\cap M$ \cite{chatzidakis2023measures}. 

    Take $K_n\in\cH'$ such that $K_n\supseteq F(\aaa)$. Since $K_n$ is regular in $K$, in particular we have $K$ and $K_n^{\alg}$ are linearly disjoint over $K_n$. Further, since $M$ is a Galois extension of $F(\aaa)$, it must be in the algebraic closure of any extension of $F(\aaa)$ including $K_n^{\alg}$. So $K\cap K_n^{\alg}=K_n$ and $K_n^{\alg}\cap M = M$, meaning $K\cap M=K_n\cap M$. 

    Then we conclude that if $K\models\varphi(\aaa)$, the same must be true of $K_n$. Since $K_n$ was chosen to be any substructure from $\cH'$ containing $F(\aaa)$, this proves $\varphi(\aaa)$ is eventually true in $\cH'$.
\end{proof}

Using the fact that we have quantifier elimination down to test formulas in both the $e$-free and $\omega$-free case, we get convergence of truth value as a corollary.

\begin{corollary}\label{cor:ctvfields}
    (Convergence of truth value). For any $\aaa\in K$ and formula $\varphi(\xx)\in\lang$, if $K\models\varphi(\aaa)$ then $\varphi(\aaa)$ is eventually true in $\cH'$.
\end{corollary}

\begin{lemma}\label{lem:fieldhomogen}
    (Homogeneity) For each $K_n\in\cH'$, $a,b\in K$, and $C\subset K$, if $a\equiv_C^K b$ then $a\equiv_{C\cap K_n}^{K_n} b$.
\end{lemma}
\begin{proof}
    Without loss of generality, suppose $C$ is algebraically closed. Since $a\equiv_C^Kb$, we have an isomorphism $\varphi:(C(a))^{\alg}\to (C(b))^{\alg}$ fixing $C$ such that $\varphi((C(a))^{\alg}\cap K)=(C(b))^{\alg}\cap K$ and $\varphi(a)=b$. 
    
    Let $\varphi': ((C\cap K_n)(a))^{\alg}\to ((C\cap K_n)(b))^{\alg}$ be the restriction of $\varphi$. Clearly $\varphi'(a)=b$, so it remains to show that $\varphi'(((C\cap K_n)(a))^{\alg}\cap K_n)=((C\cap K_n)(b))^{\alg}\cap K_n$.

    Note that $((C\cap K_n)(a))^{\alg}\cap K_n=\acl_{K_n}(C\cap K_n,a)$. Lemma \ref{lem:aclpacsubsets} gives us that this is equal to $\acl_K(C\cap K_n,a)$. By the same argument $((C\cap K_n)(b))^{\alg}\cap K_n=\acl_K(C\cap K_n,b)$. Then since $\varphi(\acl_K(C\cap K_n,a))=\acl_K(C\cap K_n,b)$, we have the desired isomorphism on the restricted domain, and therefore $a\equiv_{C\cap K_n}^{K_n}b$.
\end{proof}

\begin{theorem}
    $\cH'$ is a nice filtration of $\cH$ and therefore approximates $K$.
\end{theorem}
\begin{proof}
\begin{enumerate}
    \item \textbf{$\cH'$ covers $K$}: By construction, since we take our approximated field to be the union of the constructed chain of subfields.
    \item \textbf{Convergence of truth value}: This is Corollary \ref{cor:ctvfields}.
    \item \textbf{Stabilization of non-forking}: This is Corollary \ref{cor:pacsnf}.
    \item \textbf{Strong finite character of limit independence}: This is Corollary \ref{cor:pacsfc}.
    \item \textbf{Stabilization of algebraic closure}: This is Corollary \ref{cor:pacstabacl}.
    \item \textbf{Homogeneity}: This is Lemma \ref{lem:fieldhomogen}.
\end{enumerate}    
\end{proof}

\section{Non-Examples} \label{sec:nonex}

It is also interesting to note that there are theories that cannot be approximated by finite-dimensional substructures using this framework. The main argument in this section borrows from a similar argument for showing certain theories are not pseudofinite.

Let $\lang_{vs}$ be a two-sorted language with sorts $V$ and $K$. On $V$, there is a constant symbol $0_V$, binary functions $V^2\to V$ labeled $+_V$ and $-_V$, and for each $n$ we have an $n$-ary relation symbol $\theta_n$. On $K$ we have the language of rings: constant symbols $0_K$ and $1_K$, and binary functions $K^2\to K$ labeled $+_K,-_K,$ and $\times_K$. We also have a function symbol $\cdot:K\times V\to V$ and for each $n$ a function symbol $\pi_{n,i}:V^{n+1}\to K$.

The following lemma is adapted from \cite[Lemma 2.6]{macpherson2025omega}; the key difference being that instead of considering finite models of the theory over a finite field, we are considering finite-dimensional models over an algebraically closed field. Unless otherwise stated, ``finite dimension'' in this section can be taken to mean ``finite Morley rank''.

\begin{lemma}\label{lem:nonapprox}
    Let $(V,W,\beta)$ be a 2-sorted structure where $V,W$ are vector spaces over a field $F$ (each viewed in the two-sorted language $\lang_{vs}$ of a vector space over $F$) of dimension at least 2 and $\beta:V\times V\to W$ is bilinear. Let $\chi$ be a sentence expressing this information and $\psi$ be the sentence \[(\forall\text{ linearly independent }v_1,v_2\in V)(\forall w_1,w_2\in W)\]
    \[(\exists z\in V)(\beta(v_1,z)=w_1\land\beta(v_2,z)=w_2).\]
    Then $\chi\land\psi$ has no model of finite Morley rank.
\end{lemma}

\begin{proof}
    We may assume $F$ is algebraically closed; otherwise we have infinite Morley rank. 

    The first part of this proof closely follows the original. Suppose for a contradiction that $(V',W',\beta)$ is a model of $\chi\land\psi$ with $\dim(V')=n$ and $\dim(W')=d$ such that $n,d>1$. The assumption of finite Morley rank implies, in particular, that both $n$ and $d$ are finite. For each $v\in V'$, let $\beta_v:V'\to W'$ be the map given by $\beta_v(x)=\beta(v,x)$, and let $K_v$ denote the kernel of $\beta_v$. This map is linear and surjective, since $(V',W',\beta)\models\psi$. By the same arguments as the original proof, $\psi$ ensures that for any linearly independent $v_1,v_2\in V'$, $K_{v_1}+K_{v_2}=V'$: suppose $w_i=\beta(v_i,a_i)$ for $i=1,2$. Then $\psi$ implies that there is some $z\in V'$ such that $\beta(v_1,z)=w_1$ and $\beta(v_2,z)=w_2$, therefore $z\in(K_{v_1}+a_1)\cap (K_{v_2}+a_2)$. This shows any two cosets intersect, so for arbitrary $z\in V'$ there is some $u_1\in K_{v_1}\cap(K_{v_2}+z)$, so $u_1=u_2+z$ for some $u_2\in K_{v_2}$. Thus we can write $z=u_1-u_2\in K_{v_1}+K_{v_2}$. 

    Let $\sim$ be the equivalence relation defined on $V'$ by $v\sim w \iff K_v=K_w$, definable by the formula $\forall x\in V'(\beta(v,x)=0\leftrightarrow\beta(w,x)=0)$. Let $K\cong V'/\sim$ be the interpretable set representing this collection of kernels, and note that each element of $K$ has codimension $d$ in $V'$ and any two sum to $V'$. We also have that $K_v=K_w$ if and only if $\spann(v)=\spann(w)$\textemdash this is because $\beta_v$ and $\beta_w$ have the same kernel if and only if they are scalar multiples, so since $\beta$ is nondegenerate it follows $v$ and $w$ are scalar multiples. Then we can identify $K$ with the set of lines in $V'$ via the map taking $K_v$ to $\spann(v)$. This is a definable isomorphism of interpretable sets which preserves Morley rank, so $K$ has dimension $n-1$. 
    
    By moving to the dual space $(V')^*$ and taking annihilators, we get a set $X$ of subspaces of $(V')^*$ corresponding to the kernels $K_v$. Since the codimension of each $K_v$ in $V'$ was $d$, the dimension of each subspace in $X$ is $d$, and since the kernels all pairwise summed to the full space, the annihilators will all have pairwise intersection 0. Moving to the dual space does not change the dimension of the set, so $X$ will also have dimension $n-1$. 

    To produce a contradiction, we consider the set $\tilde{X}=\{(W,v):W\in X,v\in W-\{0\}\}$ and compute its dimension in two different ways.

    Let $\pi_1:\tilde{X}\to X$ be the projection onto the first coordinate and $\pi_2:\tilde{X}\to (V')^*-\{0\}$ be the projection onto the second coordinate. The preimage $\pi_1^{-1}$ of any $W\in X$ is the set $\{(W,v):v\in W-\{0\}\}$ which is in definable bijection with the set of non-zero vectors in $W$ and therefore has dimension $d$. Then \[\rk(\tilde{X})=\rk(X)+\rk(\pi^{-1}(W))=n-1+d\]
    Given $v\in(V')^*-\{0\}$, the preimage under $\pi_2$ is either empty (so has dimension $=-\infty$) or contains exactly one element $(W,v)$, and thus has dimension 0. In either case $\rk(\pi_2^{-1}(v))\leq 0$, so 
    \[\rk(\tilde{X})\leq\rk((V')^*-\{0\})+0=n\]

    Combining these results we have $\rk(\tilde{X})=n-1+d\leq n$, which is true if and only if $d\leq 1$. But $d$ was chosen to be at least 2, which is a contradiction.
\end{proof}

Note that a sentence $\varphi$ having no finite-dimensional model implies that a theory realizing $\varphi$ is not approximable by finite-dimensional substructures. If it were, then by Theorem \ref{thm:CTVultraproduct} the convergence of truth value condition would imply that $\varphi$ were eventually true in the approximation, meaning it would have a finite-dimensional model.

We will use this to give two examples of theories that are not approximable by finite-dimensional substructures.

\subsection{Vector spaces with $k$-linear forms}\label{subsec:klinear}

The first example of a theory that is not approximable is that of a vector space with a $k$-linear form, as studied in \cite{chernikov2024n}.

We first fix some notation: let $V$ be a vector space over $K$. We will denote a $k$-linear form on $V$ by $[-,\dots,-]_k: V^k\to K$. Using the notation from \cite[Section 2.1]{chernikov2024n}, we will write $\lozenge^{k-1}V$ to mean either $\Sym^{k-1}V$ or $\bigwedge^{k-1}V$ so as to discuss the symmetric and alternating cases uniformly. As described in \cite{chernikov2024n}, a $k$-linear form on $V^k$ gives rise to a bilinear form on $(\lozenge^{k-1}V)\times V$ which we will denote $[-,-]_2$, defined by 
\[\left[\overline{\sum_{i=1}^nk_i(v_{i,1}\otimes\dots\otimes v_{i,k-1})},v\right]_2:= \sum_{i=1}^nk_i[v_{i,1},\dots,v_{i,k-1},v]_k.\]

\begin{definition}
    \cite[Definition 2.1]{chernikov2024n} We say that the $k$-linear form $[-,\dots,-]_k$ is \textit{generic} if for any $n\in\nN$, any linearly independent elements $t_1,\dots,t_n\in\lozenge^{k-1}V$ and any $k_1,\dots, k_n\in K$ there is $w\in V$ such that $[t_i,w]_2=k_i$ for all $i\in[n]$.     
\end{definition}

We now work in an extension of the language $\lang_{vs}$ which we will denote $\lang_k$, in which we add one additional function symbol $\beta:V^k\to K$.

We define the theory $T_0$ via the following axioms:
\begin{enumerate}
    \item $K$ is a field.
    \item $V$ is a vector space over $K$, with scalar multiplication given by $\cdot:K\times V\to V$. 
    \item The relation $\theta_n(v_1,\dots,v_n)$ holds if and only if $v_1,\dots, v_n$ are linearly independent. If they are linearly independent and $w=\sum_{j=1}^n\alpha_iv_i$ is in their span, then $\pi_{n,i}(v_1,\dots,v_n,w)=\alpha_i$. Otherwise $\pi_{n,i}(v_1,\dots,v_n,w)=0$. 
    \item $\beta$ is a symmetric or alternating $k$-linear form.
\end{enumerate}

Take $T$ to be the model companion of $T_0$, axiomatized by $T_0$ along with a schema asserting that $\beta$ is generic.

\begin{definition}
    Let $\sim$ be the following definable equivalence relation: say $(v_1,\dots,v_{k-1})\sim(w_1,\dots,w_{k-1})$ if and only if $v_1\lozenge\dots\lozenge v_{k-1}=w_1\lozenge\dots\lozenge w_{k-1}$, where $\lozenge$ is taken as the relevant operation in the symmetric and alternating cases.
\end{definition}

\begin{lemma}
    To specify the behavior of the equivalence relation we need to distinguish between the two cases:
    \begin{enumerate}
        \item In the symmetric case, $\vv\sim\ww$ if and only if both products are zero or there exists a permutation $\sigma\in S_{k-1}$ and scalars $\lambda_1,\dots,\lambda_{k-1}$ such that $w_i=\lambda_i v_{\sigma(i)}$ for each $i$ and $\prod\lambda_i=1$.
        \item In the alternating case, $\vv\sim\ww$ if and only if both tuples are linearly dependent or there exists $M\in\SL_{k-1}(K)$ with $Mv_i=w_i$ for each $i$.
    \end{enumerate}
\end{lemma}
\begin{proof}
    For the symmetric case, suppose $v_1\odot\dots\odot v_{k-1}=w_1\odot\dots\odot w_{k-1}$ and both products are non-zero. Under the standard identification of $\Sym(\spann(v_1,\dots, v_{k-1}))$ with $K[x_1,\dots,x_{k-1}]$, each $v_i$ corresponds to $x_i$ and each $w_i$ to a linear polynomial $p_i$ \cite[XVI, Proposition 8.1]{lang2005algebra}. The equality becomes $x_1\dots x_{k-1}=p_1\dots p_{k-1}$, so since the polynomial ring is a unique factorization domain, each $p_i$ must be a scalar multiple of some $x_j$, so $p_i=\lambda_ix_{\sigma(i)}$ for some permutation $\sigma$. Then we have $\prod x_i=\prod \lambda_i x_{\sigma(i)}$, so it follows that $\prod\lambda_i=1$.

    For the alternating case, suppose $v_1,\dots,v_{k-1}$ and $w_1,\dots,w_{k-1}$ are both linearly independent tuples and $v_1\wedge\dots\wedge v_{k-1}=w_1\land\dots\land w_{k-1}$. Then both tuples are bases of the same $(k-1)$-dimensional subspace, so there exists a unique $M\in\GL_{k-1}(K)$ such that $Mv_i=w_i$ for each $i$. This induces a one-dimensional endomorphism $\bigwedge^{k-1}(M)$ on $\bigwedge^{k-1}(\spann(v_1,\dots,v_{k-1}))$, and the induced map is multiplication by $\det(M)$ \cite[XIX, Proposition 1.1 \& Exercise 2]{lang2005algebra}. Thus \[w_1\land\dots\land w_{k-1}=Mv_1\land\dots\land Mv_{k-1}=\det(M)(v_1\land\dots\land v_{k-1}).\] It follows from the assumption that $\det(M)=1$, so $M\in\SL_{k-1}(K)$. 
\end{proof}

Now we use this equivalence relation for a dimension argument similar to that of Lemma \ref{lem:nonapprox} in the $k$-linear context.

\begin{lemma}\label{lem:k-linear}
    Suppose $k\geq 3$. Let $\varphi_{\gen}$ be the sentence \[(\forall \vv,\ww\in V^{k-1})[(\forall z\in V)\beta(\vv,z)=\beta(\ww,z) \to  \vv\sim\ww].\] Let $\chi$ be a sentence asserting the base theory $T_0$. Then $\chi\land\varphi_{\gen}$ has no finite-dimensional models of arbitrarily large dimension.
\end{lemma}
\begin{proof}
    Suppose for a contradiction that $(V_0,K_0,\beta)\models\chi\land\varphi_{\gen}$, with $\dim(V_0)=n$. Fix a basis $e_1,\dots,e_n$ for $V_0$.

    Let $D$ be the definable set $\{\vv\in V_0^{k-1}:\vv$ linearly independent$\}$, and note that this set has dimension $n(k-1)$. Let $\Phi:D\to K_0^n$ be the map taking $\vv$ to the tuple $(\beta(\vv,e_1),\dots,\beta(\vv,e_n))$. The dimension of the image is bounded above by $\dim(K_0^n)=n$. It remains to compute the dimension of the fibers.

    First note that $\Phi(\vv)=\Phi(\ww)$ if and only if $(\forall z\in V_0)\beta(\vv,z)=\beta(\ww,z)$, since determining the form on the basis determines its behavior with respect to every vector. Further, because $(V_0,K_0,\beta)\models\varphi_{\gen}$, this implies $\vv\sim\ww$ (the reverse implication follows from the definition). Then we can calculate the dimension of the fibers via the dimension of the equivalence class $[\vv]_\sim$. We again consider the two cases separately. As $D$ is restricted to linearly independent tuples, we need not consider the zero cases.

    In the symmetric case, $[\vv]_\sim=\bigcup_{\sigma\in S_{k-1}}\{(\lambda_1v_{\sigma(1)},\dots,\lambda_{k-1}v_{\sigma(k-1)}):\prod\lambda_i =1\}$. As the dimension of a finite union is the maximum of its components, this reduces to calculating the dimension of $\{(\lambda_1v_1,\dots,\lambda_{k-1}v_{k-1}):\prod\lambda_i =1\}$, which is $k-2$.

    Then we have the following inequality calculation, writing $\Phi^{-1}(\cc)$ to represent the fiber:
    \begin{align*}
        \dim(D)-\dim(\Phi^{-1}(\cc)) &= \dim(\Phi(D))\\
        n(k-1)-(k-2) &\leq n\\
        n(k-2) &\leq k-2
    \end{align*}
    We are assuming that $k\geq 3$, so this implies $n\leq 1$. 

    In the alternating case, $[\vv]_\sim=\{(Mv_1,\dots,Mv_{k-1}):M\in\SL_{k-1}(K_0)\}$, which is in definable bijection with $\SL_{k-1}(K_0)$ and therefore has dimension $(k-1)^2-1$. Again we can set up the following inequality:
    \begin{align*}
        \dim(D)-\dim(\Phi^{-1}(\cc)) &= \dim(\Phi(D))\\
        n(k-1)-((k-1)^2-1) &\leq n\\
        n(k-2) &\leq (k-1)^2-1 \\
        n(k-2) &\leq k(k-2)
    \end{align*}
    Again since $k\geq 3$ it follows that $n\leq k$. 

    In both cases we get a uniform finite bound on the dimension of $V_0$.
\end{proof}

\begin{corollary}
    The theory of a vector space with a generic $k$-linear form as defined in \cite{chernikov2024n} is not approximable by structures of finite Morley rank.
\end{corollary}
\begin{proof}
    We only need to demonstrate that any model of $T$ satisfies $\varphi_{\gen}$. We will prove the contrapositive: suppose $\vv\not\sim\ww$, then there exists some $z\in V$ for which $\beta(\vv,z)\neq\beta(\ww,z)$. When $\lozenge\vv$ and $\lozenge\ww$ are linearly independent, this follows from genericity of $\beta$. Otherwise they are dependent, so since they are distinct both $\lozenge\vv$ and $\lozenge\ww$ are nonzero and there is some nonzero scalar $\alpha$ such that $\lozenge\ww=\alpha\lozenge\vv$. By genericity there is some $z\in V$ for which $\beta(\vv,z)=1$. Then $\beta(\ww,z)=\alpha\neq 1$. 

    Thus $T\vdash\varphi_{\gen}$, so by Lemma \ref{lem:k-linear} we cannot find models of $T$ with arbitrarily large finite $V$-dimension. Then since Morley rank in models of $T$ is controlled by the dimension of the vector sort, it follows that we cannot find models of $T$ with arbitrarily large finite Morley rank, so $T$ is not approximable.
\end{proof}

\subsection{$c$-nilpotent Lie algebras}

The second example we will consider is a two-sorted theory of $c$-nilpotent Lie algebras, studied in detail in \cite{d2024two}. We will begin with some preliminaries.
\begin{definition}
    A \textit{Lie algebra} is a vector space $L$ over a field $F$ equipped with a binary function $[\cdot,\cdot]:L^2\to L$ called a \textit{Lie bracket} satisfying the following for every $a,b,c\in L$ and $\mu\in F$.
    \begin{itemize}
        \item $[a,a] = 0$; \hfill (Alternativity)
        \item  $[a+b,c] = [a,c]+[b,c]$,  \hfill (Bilinearity) \\ 
        $[a,b+c] = [a,b]+[a,c]$,\\
        $[\mu a,b]=\mu[a,b]=[a, \mu b]$;
        \item $[a,[b,c]]+[b,[c,a]]+[c,[a,b]]=0$.\hfill (Jacobi identity)
    \end{itemize}
    A \textit{Lie subalgebra} is a subspace $U\subseteq L$ that is closed under the Lie bracket. Given two subspaces $A,B\subseteq L$, we denote the vector span by $[A,B]:=\{[a,b]:(a,b)\in A\times B\}$. A subalgebra $I\subseteq L$ is an \textit{ideal} of $L$ if $[I,L]\subseteq I$. 

    The \textit{lower central series} of $L$ is defined inductively:
    \begin{itemize}
        \item $L_1=L$;
        \item $L_{n+1}=[L_n,L]$ for $n\geq 1$. 
    \end{itemize}
    Each $L_n$ is an ideal of $L$. We say $L$ is \textit{$c$-nilpotent} if $c$ is the least integer such that \[L=L_1\supseteq L_2\supseteq\dots\supseteq L_c\supseteq L_{c+1}=0.\]
\end{definition}
\begin{definition}
    A sequence of subalgebras $(L_i)_{1\leq i\leq c+1}$ is a \textit{Lazard series} of $L$ if 
    \[L=L_1\geq L_2\geq\dots\geq L_{c+1}=0\] and \[[L_i,L_j]\subseteq L_{i+j}\] for all $i,j$, where for all $k>c$ we let $L_k=0$. 
\end{definition}

We work in the two-sorted language $\lang_{K,V,c}$. As in the previous example with $\lang_k$, this language extends the language $\lang_{vs}$. Instead of adding the function symbol $\beta$, we instead have a binary function $[\cdot,\cdot]: V^2\to V$ on $V$ and $c+1$ unary predicates $P_i$ on $V$ (for $1\leq i\leq c+1$).

The theory $T_0$ is defined as follows:
\begin{enumerate}
    \item $K$ is a field.
    \item $V$ is a $K$-vector space and $\cdot$ defines scalar multiplication.
    \item $\theta_n(v_1,\dots,v_n)$ holds if and only if $v_1,\dots, v_n$ are linearly independent. If they are linearly independent and $w=\sum_{j=1}^n\alpha_iv_i$ is in their span, then $\pi_{n,i}(v_1,\dots,v_n,w)=\alpha_i$. Otherwise $\pi_{n,i}(v_1,\dots,v_n,w)=0$. 
    \item $[\cdot,\cdot]$ is an alternating $K$-bilinear map satisfying the Jacobi identity on $V$.
    \item Each $P_i$ defines a vector subspace of $V$ such that $(P_i(V))_{1\leq i\leq c+1}$ forms a Lazard series for the Lie algebra $(V,[\cdot,\cdot])$. That is, \[V=P_1(V)\supseteq P_2(V)\supseteq\dots\supseteq P_c(V)\supseteq P_{c+1}(V)=0\] and $[P_i(V),P_j(V)]\subseteq P_{i+j}(V)$ for all $1\leq i,j\leq c+1$ where $P_{i+j}$ is the trivial subspace when $i+j>c$. 
\end{enumerate}

Let $T$ be the model companion of $T_0$, discussed in detail in \cite[Section 3]{d2024two}.

\begin{corollary}\label{cor:cnil}
    The two-sorted theory of $c$-nilpotent Lie algebras defined in \cite[Section 2]{d2024two} is not approximable by structures of finite Morley rank.
\end{corollary}
\begin{proof}
    Let $(V,K,[\cdot,\cdot])\models T$, where $V$ is the vector sort, $K$ the field sort, and $[\cdot,\cdot]$ the bilinear form. We will write $L_i=P_i(V)$. 

    Suppose $c=2k$ is even and consider the structure $(L_k/L_{k+1},L_c,[\cdot,\cdot])$, taking the Lie bracket on $(L_k/L_{k+1})^2$.
    This is well-defined on $(L_k/L_{k+1})^2\to L_c$. Let $v+L_{k+1},w+L_{k+1}\in L_k/L_{k+1}$ be coset representatives and take $[v+L_{k+1},w+L_{k+1}]=[v,w]$. Since the representatives $v,w$ are in $L_k$, we have $[v,w]\in [L_k,L_k]\subseteq L_{2k}=L_c$.

    \begin{pfclaim}
        Given $v_1,v_2\in L_k/L_{k+1}$ linearly independent and $w_1,w_2$ in $L_c$ there exists a $z\in L_k/L_{k+1}$ such that $[v_i,z]=w_i$. 
    \end{pfclaim}
    \begin{proof}
        Let $A=\vect{v_1,v_2,w_1,w_2}$, the set generated by these elements. Since $[v_1,v_2]\in L_c$ and $[v_i,w_j]=0$, this is exactly $\Span\vect{v_1,v_2}\oplus\Span\vect{w_1,w_2,[v_1,v_2]}$. We extend to $B=\Span\vect{v_1,v_2,z}\oplus\Span\vect{w_1,w_2,[v_1,v_2]}$ for $z$ linearly independent and define the form on $z$ such that $[v_i,z]=[v_1,v_2]$. Via \cite[Lemma 5.7]{d2025model}, this has the structure of a 2-nilpotent Lie algebra. Putting $z$ in the predicates for $L_k/L_{k+1}$, we can then by quantifier elimination \cite[Lemma 3.1]{d2024two} embed $B$ back into the original structure over $A$. This gives the $z$ as desired in the claim.
    \end{proof}

    Suppose instead $c=2k+1$ is odd and consider the structure $(L_k/L_{k+1},L_{c-1}/L_c,[\cdot,\cdot])$. We will prove the same claim as above, with the image of the form now considered in $L_{c-1}/L_c$.
    \begin{pfclaim}
        Given $v_1,v_2\in L_k/L_{k+1}$ linearly independent and $w_1,w_2$ in $L_{c-1}/L_c$ there exists a $z\in L_k/L_{k+1}$ such that $[v_i,z]=w_i$. 
    \end{pfclaim}
    \begin{proof}
        Let $A=\vect{v_1,v_2,w_1,w_2}$. Since $w_1,w_2$ and $[v_1,v_2]$ are no longer necessarily central, we need to consider additional elements when writing this as the span of some (coset representatives of) vectors. That is, $A=\spann\vect{v_1,v_2,w_1,w_2,[v_1,v_2],[v_i,w_j],[v_i,[v_1,v_2]]}$. Note that from the Lazard series, $w_1,w_2,[v_1,v_2]\in L_{2k}/L_{2k+1}=L_{c-1}/L_c$ and the remaining terms are in $L_{3k}/L_{3k+1}$. In the case where $c>3$ (so $k>1$) we have $3k>2k+1=c$ so $L_{3k}/L_{3k+1}$ is the trivial subspace. This is contained in $L_c$, so we reduce to the same situation as in the even case, where the additional terms are all trivial. 

        When $c=3$, we have $3k=2k+1$, so the additional terms are central but not necessarily trivial, so there is additional work to be done to account for the behavior regarding $[v_i,w_j]$ and $[v_i,[v_1,v_2]]$. 
    
        Again we extend to $B$ by adding a new element $z$, putting predicates on $z$ for $L_k/L_{k+1}$. We define the form on $z$ such that $[v_i,z]=w_i$, $[[v_1,v_2],z]=[v_1,w_2]-[v_2,w_1]$, and everything else not determined by these conditions goes to zero. Since the only terms not sent to zero are $v_1,v_2,$ and $[v_1,v_2]$, which are linearly independent, this is sufficient to specifying the form although it may not be on a basis.
        \begin{pfclaim}
            This specification of the form on $z$ satisfies the Jacobi identity.
        \end{pfclaim}
        \begin{proof}
            We need to check the Jacobi identity holds for the following combinations:
            \begin{itemize}
                \item $[[v_1,v_2],z]=[v_1,w_2]-[v_2,w_1]$
                \item $[[v_i,w_j],z]=0$
                \item $[[v_i,[v_1,v_2]],z]=0$
            \end{itemize}
    
            \begin{align*}
                [[v_1,v_2],z]&=[v_1,[v_2,z]]+[[v_1,z],v_2] & [[v_i,w_j],z] &= [v_i,[w_j,z]]+[[v_i,z],w_j] \\
                &= [v_1,w_2] + [w_1,v_2] & &= [v_i,0]+[w_i,w_j]\\
                &= [v_1,w_2] - [v_2,w_1] & &= 0+0=0
            \end{align*}
            \begin{align*}
                [[v_i,[v_1,v_2]],z] &= [v_i,[[v_1,v_2],z]]+[[v_i,z],[v_1,v_2]]\\
                &= [v_i,[v_1,w_2]-[v_2,w_1]]+[w_i,[v_1,v_2]]\\
                &= [v_i,[v_1,w_2]]-[v_i,[v_2,w_1]]+0\\
                &= 0-0+0=0
            \end{align*}
            Then as before we can embed $B$ back into the original structure over $A$, giving us $z$ satisfying the genericity claim.
        \end{proof}
    \end{proof}

    Thus for any $c\geq 2$ we can take $V=L_k/L_{2k+1}$, $W=L_c$ if $c$ is even and $L_{c-1}/L_c$ if $c$ is odd, and the Lie bracket on $V^2$, and we will satisfy the conditions set forth in Lemma \ref{lem:nonapprox}. So the theory has no pseudo-finite dimensional model. 
\end{proof}

For $p$ an odd prime and $c<p$, $\mathbf{G}_{c,p}$ denotes the \Fraisse limit of the class $\mathbb{G}_{c,p}$ of finite Lazard groups of exponent $p$ and nilpotency class at most $c$ \cite[Corollary 4.39]{d2025model}. It was shown in \cite[Proposition 2.13]{macpherson2025omega} that $\mathbf{G}_{c,p}$ is not pseudofinite for $3<c<p$, but the argument does not handle the case for $c=3$. This argument shows that the conclusion is also true for $c=3$.
\begin{proposition}
    $\mathbf{G}_{3,p}$ is not pseudofinite.
\end{proposition}
\begin{proof}
    Suppose it is pseudofinite. Let $L$ be the Lie algebra given by the Lazard correspondence with the same underlying set. Then $L$ is bi-interpretable with $\mathbf{G}_{3,p}$; hence it is also pseudofinite.

    Let $(L_i)_{i=1}^{c+1}$ be the Lazard series for both structures, and let $\psi$ be the sentence which states ``for all linearly independent $v_1,v_2\in L_1/L_2$ and $w_1,w_2\in L_2/L_3$ there exists $z\in L_1/L_2$ such that $[v_i,z]_L=w_i$.''

    The same argument as for the $c=3$ case in the proof of Corollary \ref{cor:cnil} shows that $L\models\psi$. Then it follows from \cite[Lemma 2.6]{macpherson2025omega} that $L$ is not pseudofinite, and therefore neither is $\mathbf{G}_{3,p}$.
\end{proof}

\begin{remark}
    The theory of a vector space equipped with an alternating multilinear form (Section \ref{subsec:klinear}) is $\nsop_1$ \cite[Theorem 4.14]{chernikov2024n}. In general, the theories of $c$-nilpotent Lie algebras are strictly $\nsop_4$ \cite[Theorem 4.24]{d2024two}; however in the case that $c=2$ the theory is $\nsop_1$ \cite[Corollary 3.7]{d2025model}. This is noteworthy because it indicates that while many $\nsop_1$ examples do fit into this approximation framework, approximability is not a universal property of $\nsop_1$ theories.
\end{remark}

One notable $\nsop_1$ example for which approximability is still open is the generic binary function, the model companion of the empty theory in a language consisting of one binary function symbol. We suspect this is not approximable; however the above methods for showing non-approximability do not apply as unlike in the case of a bilinear form the behavior of the binary function is not determined by its behavior on a finite subspace.

\begin{question}
    Can a model of the theory of the generic binary function be approximated by structures of finite Morley rank?
\end{question}

\section{Dimension} \label{sec:dim}
Because limit independence fails to recover Kim-independence in examples where the approximating theories are simple, we instead turn to a notion of limit dimension motivated by ideas from \cite{dobrowolski2023sets}. We would ideally like to define a notion of dimension in the limit structure that is built through the approximation such that Kim-independence in the limit structure is witnessed by a drop in the limit dimension. The construction of such a dimension is more complicated than we initially hoped, but we will use the $T^*_{\feq}$ example as a case study for what it might look like in a specific theory and then propose an initial candidate for a more general notion.

Note that since Kim-independence in strictly $\nsop_1$ theories is not base monotone, our notion of dimension must be considered over a base set $C$, where the value of the dimension may change depending on the base. That is, for a fixed set $C$ we consider $\dim_C(X)$ to be a dimension function in the typical sense and assume it is defined in some uniform way over all possible base sets.

\subsection{Case study: $T^*_{\feq}$}

We will begin by defining a general notion of approximated dimension in a parameterized theory $T_P^*$, then demonstrate how this notion behaves in the specific case of $T^*_{\feq}$. Assume that the base theory $T$ has finite Morley rank, trivial $\acl$, and eliminates quantifiers (all of which hold in the specific example $T^*_{\feq}$).

First, we need some terminology. Throughout this section, we will refer to notions of \textit{approximated dimension} and \textit{limit dimension}, where the former is determined within the approximating structures and the latter constructed as a limit of the former.

\begin{definition}\label{def:approxdim}
    Let $\M\models T$ be a structure approximated by $\cH$. For $\N\in\cH$ large enough, a base set $C\subseteq M$, and a set $X$ definable over $M$, let $\dim_{C\cap N}^\N(X\cap N)$ be a notion of dimension valued in a partially ordered abelian group $G$. We call $\dim_{C\cap N}^{\N}(X\cap N)$ the \textit{approximated dimension} of $X$ over $C$ in $\N$. 
\end{definition}

\begin{remark}
    By the convergence of truth value condition, if $X$ is definable in $\M$ then $X\cap N$ is definable in large enough $\N$ by the same formula.
\end{remark}

Typically dimension is valued in a totally ordered group. Here we specify a partial ordering as in the parameterized case the dimension will be valued as a tuple for which each coordinate is totally ordered, but the individual coordinates are independent of each other. In the $T^*_{\feq}$ case study we will ultimately convert the dimension into a totally ordered tuple.

The following means of obtaining a limit dimension as a function of the dimension in the approximation is based on the construction from \cite[Section 3]{dobrowolski2023sets}. Given $\cH,\M$, and $G$ as in the above definition, let $\cF=G^\cH$ be the set of all functions $\cH\to G$. Let $I\subset\cF$ be the subgroup \[\{f\in\cF: (\exists\N\in\cH)\;(\forall \N'\in\cH)\; (\N'\supseteq\N\implies f(\N)=0)\}.\] Consider the quotient group $S=\cF/I$. For $f\in\cF$ we will write $[f]$ to mean $f/I$. Note that $S$ has a natural partial order, with $[f]\leq[g]$ when $f(\N)\leq g(\N)$ for large enough $\N$ in $\cH$. 

\begin{definition}\label{def:limdim}
    The \textit{limit dimension} of a definable set $X$ over $C$ is \[\dim_C(X)=[f]\in S\text{ such that }f(\N)=\dim_{C\cap N}^{\N}(X\cap N).\]
\end{definition}

Now we will specify our candidate for approximated dimension in a parameterized theory under the given assumptions. Let $X$ be a set defined over $B\subseteq M$ by the formula $\varphi$ in a parameterized theory $T_P^*$. We will begin by defining $\dim_C^\N(X)$ on formulas of the form \[\bigwedge_i x_i\in\bigcap_j(D_j)_{p_j}\cap\bigcap_k(E_k)_{y_k}\] where $x_i$ are variables from the object sort, $y_k$ are variables from the parameter sort, $p_j$ are elements of the parameter sort of $B$, and each $D_j$ and $E_k$ is definable over $B$ in the unparameterized language. Because $T$ has quantifier elimination and trivial $\acl$, by \cite[Corollary 3.8]{baudisch2002generic} $T_P^*$ eliminates quantifiers and therefore by writing in disjunctive normal form every formula can be taken as a union of formulas with this form. 

\begin{definition}\label{def:paramdimp}
    Let $X$ be defined by a formula of the above form. Given a parameter $p$, we define the dimension on the objects of $X$ in a structure $\N$ with respect to $p$ as follows:
    \[\dim_p^\N(X)=\begin{cases}
        (0,0) & (D_j)_{p_j} \text{ is finite for some }j\\
        (\rk,\deg)((D_j)_{p_j}) & p=p_j\\
        (\rk,\deg)(\xx=\xx) & \text{otherwise}
    \end{cases}\]
    where $\rk$ is taken to be Morley rank in the unparameterized language and $\deg$ is the multiplicity.
    
    We then extend this definition to all formulas by imposing the condition that \[\dim_p^{\N}(X\cup Y)=\max(\dim_p^{\N}(X),\dim_p^{\N}(Y)).\]

    For the $p$-dimension on types we write $\dim_p^\N(a/B)=\min\{\dim_p^\N(\varphi):\varphi\in\tp(a/B)\}$, taking the lexicographic ordering on the pair.
\end{definition}

This allows us to consider a dimension on the objects assigned by a specific parameter as different parameters will impose different constraints on the structure of $X$.

To understand the idea behind this notion, we consider how $\dim_p^{\N}(X)$ works in $T^*_{\feq}$. For a given parameter $p$, $\dim_p^{\N}(X)$ will be the Morley rank and degree of $X$ in the unparameterized structure $\N_p$. Recall that the language of $T^*_{\feq}$ consists of a single ternary relation $E_x(y,z)$, and the theory asserts that for any parameter $p$, $E_p(y,z)$ is an equivalence relation on the object sort with infinitely many classes, each of which is infinite. In each approximating structure, we impose the restriction that the equivalence relation has only $n$ classes with respect to each parameter for some fixed finite $n$. As such, we have the following fact about Morley rank in the approximation.

\begin{fact}\label{fact:tfeq-rank}
    Working in $T^*_{\feq}$, the Morley rank of a formula $\varphi$ in $\N_p$ for some fixed parameter $p$ is exactly the number of irredundant object variables in $\varphi$ which are not fixed by equalities. This value will be the same under every parameter, and stabilizes to some fixed value for large enough $\N$ in the approximation. 
\end{fact}
The interesting behavior for the dimension in this example then comes from the degree component, which may be unbounded in the approximation.

\begin{example}
    Again we restrict our attention to the theory $T^*_{\feq}$. Consider the set $X$ defined over $C=\{p,c\}$ by the formula $\neg E_p(x,c)$. Let $\M_n\in\cH$ with $n>1$; that is, $\M_n$ represents a parameterized equivalence relation with $n$ classes for each parameter. In the unparameterized structure $(\M_n)_p$, $X$ denotes the subset of $(\M_n)_p$ excluding a single equivalence class (the $E_p$-class of $c$). $X$ then has Morley rank 1 and Morley degree $n-1$, as by specifying an $E_p$ class for $x$ (of which there are $n-1$ choices) we can partition $X$ into $n-1$ infinite pieces. As we move up through the approximation, $n$ increases unbounded so the Morley degree will also increase accordingly. 

    Intuitively, this captures the element of forking in the limit that cannot be witnessed in the approximation. In the limit structure there are infinitely many classes with respect to each parameter, so we can take an indiscernible sequence moving between equivalence classes and get inconsistency via a formula specifying the equivalence class. However, in the approximation there are only finitely many classes so to be indiscernible a sequence must be contained in a single $E_p$-class, meaning the inconsistency that comes from fixing a class cannot be witnessed downstairs. This discrepancy between finite and infinitely many classes is what is captured by the unbounded multiplicity.

    Note that the multiplicity in the unparameterized structure is dependent on the choice of $p$. In this example, if we instead consider the dimension with respect to some other parameter $q$, then since $\varphi$ says nothing about the structure of $X$ under the parameter $q$ the degree is maximal with respect to $q$\textemdash that is, the degree will be $n$.
\end{example}

\begin{remark}
    In $T^*_{\feq}$ the rank component is always the same for every parameter and the interesting behavior occurs in the multiplicity coordinate where we can see unbounded growth through the approximation. This is not necessarily true in general\textemdash unbounded multiplicity is a property of the approximation of $T^*_{\feq}$ that is not shared by every parameterized theory. Likewise, the rank may behave in more complicated ways in other examples.
\end{remark}

We will sometimes use the notation $\rk_p$ or $\deg_p$ for the rank and degree components of $\dim_p$. We consider $\dim_p$ as a partially ordered tuple, totally ordered in each coordinate. The ordering on the rank coordinates is that of the natural numbers (as each $\N\in\cH$ has finite Morley rank) and the ordering on the degree components is determined by an infinite difference\textemdash that is, we say $\deg_p(X)<\deg_p(Y)$ if and only if $\deg_p(Y)-\deg_p(X)$ is infinite. Note that under this definition the approximating structures cannot witness a drop in the degree components; however, an infinite difference may occur in the limit in cases where multiplicity is unbounded (and therefore infinite in the limit).

\begin{definition}
    For a tuple $a$, let $\rk_=^\N(a/B)$ be the Morley rank of $a$ over $B$ in $\N$ taken in the language of equality. Let $X$ be defined by the formula $\varphi(x,b)$. Then we define \[\rk_=^\N(X)=\rk_=^\N(\varphi(x,b))=\max(\rk_=^\N(a/b):a\in X).\]
\end{definition}
\begin{remark}
    Since $\rk_=$ tracks only equalities over the defining set, in $\N$ large enough to witness all relevant parameters $b$ it will always agree with the corresponding rank in $\M$. Under the assumption that $T$ has trivial $\acl$, all algebraicity comes from equality so $\rk_=$ is sufficient to capture any algebraic dependence in the dimension. In the case where $T$ has non-trivial algebraicity, the means of capturing this behavior in the dimension will likely be more complicated.
\end{remark}

\begin{definition}\label{def:paramdim}
    Let $X$ be definable in $M$ by the formula $\varphi(x,b)$. Let $\pi_O(X)$ be the projection to the object coordinates. Then for large enough $\N$ the approximated dimension of $X$ in $\N$ over a base set $C$ with finitely many parameters is given by:
    \[\dim_C^{\N}(X)=(\rk_=^\N(X\cap N/b),(\dim_p^\N(\pi_O(X)\cap N))_{p\in C}).\]
    Similarly for types we let \[\dim_C^{\N}(a/B)=(\rk_=^\N(a/B),(\dim_p^\N(a/B))_{p\in C}).\]
    In both cases the limit dimension is then determined as in Definition \ref{def:limdim}.
\end{definition}
\begin{remark}
    In this definition we specify that $C$ has finite parameter sort to prevent the case where we are comparing tuples of infinite length. This will be necessary in the $T^*_{\feq}$ example in order to handle independence between the structures imposed by different parameters. 
\end{remark}

This gives us a general candidate for limit dimension in a parameterized theory, valued in a partially ordered set of functions. We now define the following independence notion, representing a drop in limit dimension. 
\begin{definition}\label{def:dimdrop}
    We say $a\ind^{\dim}_Cb$ if for every $C'\subseteq C$ with finite parameter sort, \[\dim_{C'}(a/Cb)=\dim_{C'}(a/C).\] That is, the dimension of $a$ does not drop over $b$ with respect to any finite subset of parameters in $C$.
\end{definition}

Note that both $\dim_C^{\N}$ and $\dim_C$ are valued in partially ordered sets; the former due to the independence of the parameter variables (e.g, the dimension of $X$ may be less than the dimension of $Y$ with respect to $p_1$ but greater with respect to $p_2$) and the latter due to the partial ordering of functions. In general the fact that the image is not totally ordered may cause issues in the case of incomparables. In particular, the dimension of a type is typically defined as the infimum of the dimensions of all formulas it implies; however in a partial ordering there is the possibility that this may not meaningfully exist.

For the remainder of this subsection we restrict our attention to $T^*_{\feq}$. In this case, when $C$ has finitely many parameters we can convert the image of $\dim_C$ to a lexicographically ordered triple of natural numbers while preserving the partial ordering, meaning that we can in fact consider the values in a totally ordered set.

\begin{definition}
    Let $\varphi$ be a complete formula in $\lang_{\feq}$ and $p$ an element of the parameter sort. We say that a collection of object variables $x_1,\dots,x_n$ is \textit{$p$-independent in $\varphi$} if $\varphi$ does not imply $E_p(x_i,x_j)$ for any $i\neq j$.
    We say that the $E_p$-class of a variable $x_i$ is \textit{bounded} if there are only finitely many choices of $E_p$-class for $x_i$ that are consistent with $\varphi$. Otherwise $x_i$ has \textit{unbounded $E_p$-class}.
\end{definition}

\begin{remark}\label{rem:p-ind}
    The idea behind $p$-independent variables is to restrict attention to those for which fixing the $E_p$-class of one will \textit{not} fix the class of any other variable. For formulas in general (which may not be complete), it is likely necessary to impose an additional condition for $p$-independence along the lines of ``$\varphi$ does not imply any finite disjunction of positive equivalences on more than one $x_i$ (with parameters from $M$)''. This is because $\varphi$ may imply a dependence between $x_i$ and $x_j$ of a form such as $E_p(x_i,b)\lor E_p(x_j,b')$, in which $x_i$ and $x_j$ may not be equivalent but nonetheless fixing the class of one may fix the class of the other. However since complete formulas make a decision on every disjunction, in the complete case such a dependence would already imply a fixed class for one or both of $x_i$ and $x_j$.
\end{remark}

\begin{lemma}\label{lem:feq-poly-degree}
    Let $\varphi(\xx)$ be a complete formula that defines the set $X$. Then when computing $\dim_C(X)$, for any parameter $p\in C$ we can consider $\deg_p(X)$ as a polynomial function on the number of classes in the approximating structures whose polynomial degree equals the number of $p$-independent variables of $\varphi$ with unbounded $E_p$-class.
\end{lemma}
\begin{proof}
    Since $\varphi$ is complete it makes a choice for every disjunct, so writing in disjunctive normal form we can consider $\varphi$ as a conjunction of literals. By definition of $\deg_p(\varphi)$ we can take $\varphi$ to be conjunction of literals in the unparameterized language, considering only the conjuncts that impact the structure dictated by the parameter $p$. Each literal can only say that a variable $x_i$ is (in)equal or (in)equivalent (via $p$) to a constant or another variable. By Fact \ref{fact:tfeq-rank}, the rank $\rk_p$ of any definable set in the unparameterized language is exactly the number of irredundant variables in the defining formula that are not fixed by equalities, and is independent of the choice of $p$. Now, $\deg_p$ represents the number of $\rk_p$ subsets into which $X$ can be partitioned and the only way to obtain $\rk_p$ subsets via conjunction with $\varphi$ is to fix the equivalence class of a variable that does not already have a fixed class. The multiplicity is then exactly the number of ways to choose $E_p$-classes for all the variables.

    Under a fixed parameter $p$ we restrict our attention to the $p$-independent variables of $\varphi$. Per Remark \ref{rem:p-ind}, this guarantees that there will be no cases in which fixing the class of one $x_i$ fixes the class of any other $x_j$ (though the choice of class for $x_i$ may still reduce the number of choices for $x_j$).

    Denote the $p$-independent variables of $\varphi$ by $x_1,\dots,x_m$ and further assert that $x_1,\dots,x_\ell$ are variables with unbounded $E_p$-class while $x_{\ell+1},\dots,x_m$ are the variables with bounded $E_p$-class. Assuming no additional conditions have been imposed at this point, there are $n-k_i$ choices of $E_p$-class for each of $x_1,\dots,x_\ell$ where $n$ is the number of classes in the approximating structure $\M_n$ and $k_i$ is the (finite) number of $E_p$-classes that \textit{cannot} be chosen for $x_i$ without contradicting $\varphi$, and exactly one choice of $E_p$-class for each of $x_{\ell+1},\dots,x_m$ (since $\varphi$ is complete and therefore fixes the class).

    If there is no dependence between any of $x_1,\dots,x_m$ then the number of choices for the $E_p$-class of all the variables is just the product of the number of choices for each individual variable:
    \[\prod_{1\leq i\leq\ell}(n-k_i).\] 
    Note that each of $x_{\ell+1},\dots,x_m$ contributes a factor of 1, so we need only consider the variables with unbounded $E_p$-class. As $n$ grows through the approximation, in the limit we clearly see that this becomes a polynomial in $n$ with degree $\ell$.

    Now we need to consider the possibility of one or more dependences between variables. The only case in which the choice of class for some $x_i$ can affect the number of choices remaining for some $x_j$ is if $\varphi$ implies $\neg E_p(x_i,x_j)$. Any other dependence between variables would have to take the form $\bigwedge(\psi_i(x_i))\to\psi_j(x_j)$, where each $\psi_i$ makes an assertion about the $E_p$-class of $x_i$. This we can rewrite as a disjunction $\bigvee(\neg \psi_i(x_i))\lor \psi_j(x_j)$ and by completeness $\varphi$ already decides the outcome. Thus any dependence between $p$-independent variables for the number of choices of $E_p$-class is reduced to this binary case. We also need only consider the case when both $x_i$ and $x_j$ have unbounded $E_p$-class.

    Without loss of generality, suppose $\varphi$ implies $\neg E_p(x_1,x_2)$. Suppose we have already chosen $E_p(x_1,b)$ for some $b\in M$. Then, letting $n-k_2$ represent the number of choices for $x_2$ before fixing $E_p(x_1,b)$, we have two cases: either $E_p(x_2,b)$ is consistent with $\varphi$ and the number of choices remaining for $x_2$ is $n-k_2-1$, or it is inconsistent and the number of choices is unchanged. There can be at most $k$ distinct $E_p$-classes for which $E_p(x_2,b)$ is inconsistent with $\varphi$ (for some finite $k$), so the number of ways to choose both $x_1$ and $x_2$ will have the form: \[k(n-k_2)+(n-k_1-k)(n-k_2-1).\]
    Note that this is a polynomial in $n$ of degree 2, as desired. It follows that extending to more $p$-independent variables we will always end up with a polynomial of degree $\ell$ as dependences of the form $\neg E_p(x_i,x_j)$ only reduce the number of choices for $x_j$ by a finite amount, so the number of choices for $x_j$ will remain unbounded for large enough $n$.
\end{proof}

\begin{remark}\label{rem:feq-poly-degree}
    For complete formulas, Lemma \ref{lem:feq-poly-degree} gives an explicit description of the limit of $\deg_p$ as a polynomial function on $n$. Since we consider $\deg_p(X)<\deg_p(Y)$ if and only if $\deg_p(Y)-\deg_p(X)$ is infinite, the only relevant data for comparing $\deg_p$ of two different complete formulas is the polynomial degree (from the proof we see that the leading coefficient is always one). So we can in fact consider $\deg_p$ to be just the polynomial degree.
\end{remark}

\begin{remark}
    In general, $\aleph_0$-categoricity allows us to write any formula $\psi$ as a disjunction of complete formulas $\varphi_i$. Then $\deg_p(\psi)$ is bounded above by the sum of $\deg_p(\varphi_i)$ for each $i$. This is only an upper bound and doesn't give an exact value for the multiplicity of $\psi$ with respect to $p$, but for our purposes this is enough. In fact, for $\ind^{\dim}$ our notion of independence we are working explicitly with the dimension of complete types and therefore only need to consider the case of complete formulas. 
\end{remark}

Now we will conclude that in $T^*_{\feq}$ $\dim_C$ can be taken as valued in a totally ordered set.

\begin{lemma}\label{lem:feq-pi}
    Suppose $C$ has finitely many parameters. Then there is an order-preserving function $\pi:\im(\dim_C)\to\omega^3\cup\{-\infty\}$ on complete formulas in $T^*_{\feq}$. 
\end{lemma}
\begin{proof}
    Let $X$ be a set definable by the complete formula $\varphi$. First note that $\rk_=(X)$ captures exactly the number of irredundant variables in $\varphi$ not bound to a finite set, so it stabilizes to a fixed value on a cone in the approximation. Moving to a higher cone if necessary, we can consider the dimension to be constant on this coordinate. 

    Now we consider the tuples $(\rk_p,\deg_p)(X)$ for each parameter $p\in C$. Again from Fact \ref{fact:tfeq-rank}, $\rk_p$ is independent of the choice of $p$ and will stabilize to a fixed value on a cone, so we can take the single value as our second coordinate.

    From Lemma \ref{lem:feq-poly-degree} we have that each $\deg_p$ component of the limit dimension can be expressed as a polynomial function on the number of $E_p$-classes in the approximating structure, and by Remark \ref{rem:feq-poly-degree} the relevant data is in fact contained only in the polynomial degree of this function. Thus we let $d_p(X)$ represent the polynomial degree of $\deg_p(X)$ for each $p\in C$, and take our third coordinate to be $\sum_{p\in C}d_p(X)$. Since $C$ has only finitely many parameters, this will be finite.

    Now we define $\pi$ such that:\[\pi(\dim_C(X))=\left(\rk_=(X),\rk_p(X),\sum_{p\in C}d_p(X)\right),\] with the lexicographic ordering. It follows immediately from the definition that given any definable $X$ and $Y$, if $\dim_C(X)\leq\dim_C(Y)$, then $\pi(\dim_C(X))\leq\pi(\dim_C(Y))$.
\end{proof}

\begin{corollary}
    Limit dimension on types in $T^*_{\feq}$ is invariant of choice of filtration.
\end{corollary}
\begin{proof}
    Per the proof of the preceding lemma, the data captured by $\dim_C(X)$ is: the number of irredundant variables of the defining formula not fixed by equalities, the number of irredundant object variables not fixed by equalities, and the sum total of $p$-independent object variables with respect to each $p\in C$ without fixed $E_p$-class. These are all features of the formula itself, not a specific filtration of $\cH$, so we will get the same output for large enough $\N$ in any nice filtration of $\cH$.
\end{proof}

\begin{remark}
    The reason for taking the sum of the degree components is the independence of the parameters. Each $d_p$ tracks the amount of $p$-independent variables without a fixed $E_p$-class, but the complexity only depends on the total number, not on the distribution among specific parameters.
\end{remark}

\begin{theorem}\label{thm:feq-dimkim}
    In $T^*_{\feq}$, $\ind^{\dim}=\ind^K$.
\end{theorem}
\begin{proof}
    From \cite[Section 6.3]{chernikov2016model}, $A\ind^K_CB$ in $T^*_{\feq}$ if and only if $A\cap B\subseteq C$ and for every parameter $p\in C$, the $E_p$-classes of $A$ and $B$ are disjoint over $C$. By Definition \ref{def:dimdrop}, $A\ind^{\dim}_CB$ if and only if for every $C'\subseteq C$ with finite parameter sort, $\dim_{C'}(A/BC)=\dim_{C'}(A/C)$. Considering $\dim_{C'}$ as a triple $(\rk_=,\rk_p,\sum_{p\in C'} d_p)$, this is equivalent to the following assertions:
    \begin{itemize}
        \item $\rk_=(A/BC)=\rk_=(A/C)$,
        \item $\rk_p(A/BC)=\rk_p(A/C)$, and
        \item $\sum_{p\in C'} d_p(A/BC)=\sum_{p\in C'}d_p(A/C)$.
    \end{itemize}
    The first point is equivalent to having $A$ and $B$ disjoint over $C$, the first condition for $\ind^K$. Note that for each parameter $p$, $d_p(A/BC)<d_p(A/C)$ if and only if there is some $a\in A$ with unbounded $E_p$-class in $\tp(A/C)$ and bounded $E_p$-class in $\tp(A/BC)$, which can only occur in the case of an $E_p$-equivalence between $A$ and $B$ not witnessed in $C$. Since the degree is monotone, equality of the sums implies equality of each component, so this equality is equivalent to having the $E_p$-classes of $A$ and $B$ disjoint over $C$ for every $p\in C'$. 
    Then, since $A\ind^{\dim}_CB$ holds if and only if the equality is true for \textit{every} $C'\subseteq C$ with finite parameter sort, the $E_p$-classes are disjoint over $C$ for every $p\in C$. Therefore $\ind^{\dim}=\ind^K$.
\end{proof}

\subsection{Bounded vs. unbounded multiplicity}

At this point we will briefly mention the vector space example $T_\infty$ and how it necessarily differs from $T^*_{\feq}$. The key point is that while in $T^*_{\feq}$ the interesting behavior in the dimension comes from unbounded multiplicity, this cannot occur in the Granger example.

\begin{theorem}\label{thm:bdd-mult}
    Let $(K,V)\models T_\infty$. For any definable subset $X$ of $(K,V)$, $X$ has bounded multiplicity in the approximation by finite dimensional subspaces.
\end{theorem}
\begin{proof}
    If $X$ is a set of field elements, then since the field is constant across the approximating family the multiplicity cannot grow as we move through the approximation. If $X$ is finite-dimensional, then for large enough $\N$ it can be put in definable bijection with a subset of the field, and the multiplicity is again bounded. This leaves the case where $X$ is infinite-dimensional.

    Suppose $\varphi(x;\bar\alpha,\vv)$ defines the infinite-dimensional set $X\subseteq V$, for $\bar\alpha\in K,\vv\in V$. Using quantifier elimination to write in disjunctive normal form and the fact that the multiplicity of a disjunction is bounded by the sum of the multiplicity of the disjuncts, we can reduce to the case where $\varphi(x;\bar\alpha,\vv)$ is a conjunction of literals. Without loss of generality we may replace $\vv$ with a basis of their spanning set and consider any additional vector parameters as terms. If $\varphi$ contains a conjunct of the form $\neg\theta_n(\vv,x)$ then $X$ is finite-dimensional, so we need only consider literals of the form $\theta_n(\vv,x)$ and (in)equalities of terms $t(x,\bar\alpha,\vv)$. This also simplifies the term structure as the coordinate functions will not produce any interesting behavior on $x$ with $\vv$, so our terms will be constructed by linear combinations of vectors in $\vv$ along with applications of the bilinear form. By adding finitely many field elements to the parameters $\bar\alpha$ to account for scalar multiplication, we can reduce any equalities incorporating the bilinear form to the form $\beta(x,v_i)=\alpha_i$.
    
    \begin{pfclaim}
        If $\varphi'=\varphi\land\beta(x,v)=\alpha$, then for large enough $\N$ in the approximation $\rk_\N(\varphi')=\rk_\N(\varphi)-1$. 
    \end{pfclaim}
    \begin{proof}
        Suppose $\N$ is large enough in the approximation so that the dimension is at least $k$. 
        
        First we consider the symmetric case. For $V$ an $n$-dimensional space we take $v_1,\dots,v_n$ to be an orthonormal basis of $V$. Let $Y=\psi(V)$; that is $Y=\{x\in V:\beta(x,v)=\alpha\}$. Without loss of generality, suppose $v=v_1$. Then since every vector element in $V$ can be written in the form $\sum_{i=1}^n\gamma_iv_i$, the restriction $\beta(x,v)=\alpha$ simplifies to $\gamma_1=\alpha$. Now consider the intersection $X\cap Y$. Since $Y$ fixes one additional degree of freedom, the dimension of $X\cap Y$ is $\dim(X)-1$. 

        In the alternating case we can make a near identical argument using a symplectic basis to get that the dimension also decreases by 1.
    \end{proof}

    Thus for large enough $\N$ in the approximation, any conjunct with $\varphi$ will drop the rank, so it is not possible for the multiplicity to grow unbounded in the approximation when $X$ is a set of singletons.

    In higher-variable cases, the additional complexity reduces to consideration of the formula $\beta(x,y)=\alpha$. By a similar argument as above, we get that for $x=\sum_{i=1}^n\gamma_iv_i$ and $y=\sum_{i=1}^n\delta_iv_i$ (in the symmetric case) this reduces to $\alpha=\sum_{i=1}^n\gamma_i\delta_i$, so the dimension again drops by one. Again the alternating case is similar.
\end{proof}

We could consider an unparameterized notion of the dimension from Definition \ref{def:paramdim} as a candidate for $T_\infty$, in which instead of passing to the rank and degree with respect to each element of a parameter sort we simply consider the rank and degree within the approximating structures as is. However, from Theorem \ref{thm:bdd-mult} we know that the degree component will be a finite constant for large enough $\N$ and as such cannot witness a drop in dimension because the limit difference between any two values will always be finite. Therefore the only meaningful part of the limit dimension would be the limit of the rank component, which is able to capture $\Gamma$-independence in the limit \cite[Corollary 8.4]{dobrowolski2023sets} but is insufficient to capture Kim-independence.

\subsection{Axiomatic notion of dimension}

Now that we have a notion of limit dimension for $T^*_{\feq}$ that recovers Kim-independence in the limit, we would like to extend this to other examples. The proof of Theorem \ref{thm:feq-dimkim} uses the explicit description of Kim-independence in $T^*_{\feq}$, which does not generalize. We will now propose an abstract axiomatic notion of dimension and show that for an approximated structure with a notion of dimension satisfying these axioms, we can recover almost all of the properties necessary to prove it is in fact Kim-independence.

\begin{definition}\label{def:approxdimaxioms}
    Let $\M\models T$ be a structure approximated by $\cH$ and let $\dim_{C\cap N}^{\N}(X\cap N)$ be the approximated dimension from Definition \ref{def:approxdim}. Suppose $\dim_{C\cap N}^{\N}(X\cap N)$ satisfies the following axioms for large enough $\N$:
    \begin{enumerate}
        \item Invariance: If $a\equiv_{C\cap N}^{\N} a'$ then $\dim_{C\cap N}^{\N}(\varphi(x,a))=\dim_{C\cap N}^{\N}(\varphi(x,a'))$
        \item Algebraicity: If $X$ is finite non-empty then $\dim_{C\cap N}^{\N}(X)=0$, and $\dim_{C\cap N}^{\N}(\varnothing)=-\infty$.
        \item Union: $\dim_{C\cap N}^{\N}(X\cup Y)=\max\{\dim_{C\cap N}^{\N}(X),\dim_{C\cap N}^{\N}(Y)\}$.
        \item Fibration: If $f:X\to Y$ is an interpretable map such that $\dim_{C\cap N}^{\N}(f^{-1}(y))\geq d$ for all $y\in Y$, then $\dim_{C\cap N}^{\N}(X)\geq \dim_{C\cap N}^{\N}(Y)+d$.
        \item Additivity: $\dim_{C\cap N}^{\N}(ab/A)=\dim_{C\cap N}^{\N}(a/Ab)+\dim_{C\cap N}^{\N}(b/A)$. 
        \item Eventual uniform definability: If $f:\cH\to G$ is the function given by $f(\N)=\dim_{C\cap N}^{\N}(\varphi(x,a))$ then there exists some formula $\psi_{\varphi,f}\in\tp(a)$ and some $\N'\in\cH$ such that in every $\N\in\cH$ with $\N\supseteq\N'$ we have $\dim_{C\cap N}^{\N}(\varphi(x,a'))=f(\N)$ for all $a'\models\psi_{\varphi,f}.$
    \end{enumerate}
    Then we may consider the limit dimension defined over this approximated dimension in the sense of Definition \ref{def:limdim}.
\end{definition}

Note that we cannot require a total ordering of $G$, as in the parameterized case study the image of dimension is a set of tuples that is totally ordered in each coordinate but not necessarily as a whole. We will try to generalize the argument from the $T^*_{\feq}$ case to show when we can convert the limit dimension into a totally ordered set.

Let $\pi:\im(\dim_C)\to \omega^k\cup \{-\infty\}$ be an order-preserving linear map from the image of $\dim_C$ to a $k$-tuple of natural numbers with the lexicographic ordering, for some fixed finite $k$. Then we can take the values of $\dim_C(X)$ to be in this totally ordered set, while preserving the following axioms inherited from the approximation:

\begin{lemma}\label{lem:limdimaxioms}
    Suppose $\pi:S\to\omega^k\cup\{-\infty\}$ for some fixed finite $k>0$ satisfies the following conditions:
    \begin{enumerate}
        \item Linearity: for all $[f],[g]\in S$, $\pi([f]+[g])=\pi([f])+\pi([g])$.
        \item Order-preserving: for all $[f],[g]\in S$, $[f]\leq[g]\implies\pi([f])\leq\pi([g])$.
        \item Algebraic normalization:  $\pi(-\infty)=-\infty$ and $\pi(0)=0$. 
    \end{enumerate}

    Then both $\dim_C(X)$ and $(\pi\circ\dim_C)(X)$ satisfy the following axioms.
    \begin{enumerate}
        \item Invariance: If $a\equiv_C^{\M} a'$ then $\dim_C(\varphi(x,a))=\dim_C(\varphi(x,a'))$
        \item Algebraicity: If $X$ is finite non-empty then $\dim_C(X)=0$, and $\dim_C(\varnothing)=-\infty$.
        \item Additivity: $\dim_C(ab/A)=\dim_C(a/Ab)+\dim_C(b/A)$. 
        \item Definability: If $\dim_C(\varphi(x,a))=d$ then there exists some formula $\psi_{\varphi,d}\in\tp(a)$ such that $\dim_C(\varphi(x,a'))=d$ for all $a'\models\psi_{\varphi,d}.$
    \end{enumerate}
\end{lemma}
\begin{proof}
    The arguments largely follow directly from the definition and the axioms for dimension in the approximation. For this proof, we will take $\dim_C(X)$ to be the value in $S$ and $\pi(\dim_C(X))$ the corresponding value in $\omega^k\cup\{-\infty\}$. 
    
    \underline{Invariance}: Let $a\equiv_C^\M a'$. By homogeneity of $\cH$, $a\equiv_{C\cap N}^{\N}a'$ for any $\N\in\cH$ containing $a$ and $a'$, so invariance of dimension in the approximation asserts that $\dim_{C\cap N}^{\N}(\varphi(x,a))=\dim_{C\cap N}^{\N}(\varphi(x,a'))$ in any such $\N$. Then by definition $\dim_C(\varphi(x,a))=\dim_C(\varphi(x,a'))$, and it follows that $\pi(\dim_C(\varphi(x,a)))=\pi(\dim_C(\varphi(x,a')))$.

    \underline{Algebraicity}: Suppose $X$ is finite and non-empty. Then $X\cap N$ is finite for all $\N\in\cH$, and for large enough $\N$ is non-empty, so algebraicity in the approximation means that $\dim_C(X)=[x\mapsto 0]$. 
    
    For all $\N\in\cH$, $\dim_C^{\N}(\varnothing)=-\infty$, so $\dim_C(X)=[x\mapsto -\infty]$. By algebraic normalization, $\pi$ maps the constant $-\infty$ function to $-\infty$ and the constant 0 function to 0, so algebraicity holds. 
    
    \underline{Additivity}: Let $[f]=\dim_C(ab/A)$, $[g]=\dim_C(a/Ab)$, and $[h]=\dim_C(b/A)$. Then by additivity in the approximation, for large enough $\N\in\cH$ \[f(\N)=\dim_{C\cap N}^{\N}(ab/A)=\dim_{C\cap N}^{\N}(a/Ab)+\dim_{C\cap N}^{\N}(b/A)=g(\N)+h(\N).\] So $[f]=[g]+[h]$, and additivity holds. By linearity of $\pi$, $\pi([f])=\pi([g]+[h])=\pi([g])+\pi([h])$, so it additionally holds for $\pi\circ\dim_C$. 

    \underline{Definability}: Let $\dim_C(\varphi(x,a))=[f]\in S$ and let $\N'\in\cH$ and $\psi_{\varphi,f}$ be as given by eventual uniform definability of $\dim_C^{\N}$. Suppose $a'\models\psi_{\varphi,f}$. Then for all $\N\in\cH$ with $\N'\subseteq\N$, $\dim_{C\cap N}^{\N}(\varphi(x,a'))=f(\N)$. By definition, $\dim_C(\varphi(x,a'))=[f]=\dim_C(\varphi(x,a))$, and it follows that $\pi(\dim_C(\varphi(x,a')))=\pi(\dim_C(\varphi(x,a)))=\pi([f])$.
\end{proof}

\begin{remark}
    Union holds for $\dim_C$ under the assertion that $\max([f],[g])$ is the coordinate-wise maximum. It holds for $\pi\circ\dim_C$ under the additional assertion that $\pi$ preserves coordinate-wise maximum; however this condition is not met in the $T^*_{\feq}$ case.
\end{remark}

From now on we will drop $\pi$ from the notation and simply consider $\dim_C(X)$ as valued in the totally ordered set $\omega^k\cup\{-\infty\}$.

\begin{definition}
    Let $\ind^{\dim}$ be the independence relation representing a drop in limit dimension; that is:
    \[A\ind^{\dim}_CB \iff \dim_C(A/C)=\dim_C(A/B)\]
\end{definition}

\begin{theorem}\label{thm:dimprops}
    If a structure $\M\models T$ approximated by $\cH$ has limit dimension $\dim_C(X)$, then $\ind^{\dim}$ satisfies the following axioms: symmetry, monotonicity, strong finite character, and extension.
\end{theorem}
\begin{proof}
   
    \underline{Symmetry}: follows from additivity. Suppose $a\ind^{\dim}_Cb$. Then $\dim_C(a/C)=\dim_C(a/Cb)$. Consider $\dim_C(ab/C)$. By additivity of limit dimension, 
    \[\dim_C(ab/C)=\dim_C(a/Cb)+\dim_C(b/C)\] and \[\dim_C(ab/C)=\dim_C(b/Ca)+\dim_C(a/C)\]
    By our assumption that $\dim_C(a/C)=\dim_C(a/Cb)$ we can cancel, leaving $\dim_C(b/C)=\dim_C(b/Ca)$. Then by definition $b\ind^{\dim}_Ca$.

    \underline{Monotonicity}: Suppose $aa'\ind^{\dim}_Cbb'$, so $\dim_C(aa'/Cbb')=\dim_C(aa'/C)$. By the definition of dimension on types it follows that $\dim_C(aa'/C)=\dim_C(aa'/Cb)$, since $\tp(aa'/C)\subseteq\tp(aa'/Cb)\subseteq\tp(aa'/Cbb')$. Then $aa\ind^{\dim}_Cb$, so by symmetry and the same argument we get $a\ind^{\dim}_Cb$.

    \underline{Strong finite character}: Suppose $a\nind^{\dim}_Cb$. By symmetry, $b\nind^{\dim}_Ca$ so $\dim_C(b/C)\neq\dim_C(b/aC)$. Let $d=\dim_C(b/aC)$, so $d<\dim_C(b/C)$. Let $\varphi(x,a)\in\tp(b/Ca)$ such that $\dim_C(\varphi(x,a))=d$. Then by definability there is a formula $\psi_{\varphi,d}$ such that if $a'\models\psi_{\varphi,d}$ then $\dim_C(\varphi(x,a'))=d$. 

    Let $\theta(y)$ be the formula $\varphi(b,y)\land\psi_{\varphi,d}(y)\in\tp(a/Cb)$. Suppose $a'\models\theta$. Then since $\varphi(x,a')\in\tp(b/Ca')$ we have \[\dim_C(b/Ca')\leq\dim_C(\varphi(x,a'))=[d]<\dim_C(b/C)\]
    so $\dim_C(b/Ca')<\dim_C(b/C)$. In particular, $\dim_C(b/Ca')\neq\dim_C(b/C)$, so $b\nind^{\dim}_Ca'$ and by symmetry $a'\nind^{\dim}_Cb$. Thus strong finite character holds via the formula $\theta$. 
    
    \underline{Extension}: Suppose $\dim_C(a/Cb)=\dim_C(a/C)=d$. Let $p(x)=\tp(a/Cb)\cup\{\neg\varphi(x,\bb):\bb\in Cbb'$ and $\dim(\varphi(x,\bb))<d\}$. Suppose $p(x)$ is inconsistent. Then by compactness $\tp(a/Cb)$ implies a finite disjunction of formulas $\varphi(x,\bb)$ with dimension $<d$, so $\tp(x/Cb)$ must itself have dimension $<d$, a contradiction.    
\end{proof}

\begin{remark}
    If, in addition to the properties in Theorem \ref{thm:dimprops}, $\ind^{\dim}$ satisfies the Independence Theorem, this is sufficient to show that it strengthens Kim-independence \cite[Theorem 6.1]{chernikov2023transitivity}.
\end{remark}

\subsection{Next steps}

The case study of $T^*_{\feq}$ provides various insights into what might and might not work for a general notion of limit dimension. In particular, the proposed axiomatic notion of dimension covers most of the properties necessary to prove $\nsop_1$, with the notable exception of the independence theorem. 
\begin{question}
    What additional axioms of approximated dimension are needed to obtain the independence theorem for $\ind^{\dim}$?
\end{question}

One approach we tried to answer this question was to include as an axiom the $S_1$ property, which asserts that if there exists an indiscernible sequence $(b_i)_{i<\omega}$ and a formula $\varphi(x,y)$ such that
\begin{itemize}
    \item $\dim_{C\cap N}^{\N}(\varphi(x,b_1)\land\varphi(x,b_2))\leq d$ and
    \item $\dim_{C\cap N}^{\N}(\theta\land\varphi(x,b_i))>d$ for each $i$,
\end{itemize}
then $\dim_{C\cap N}^{\N}(\theta)>d$.

If we could obtain this property in the limit, we would be able to prove the independence theorem directly; however unlike the other axioms this property of limit dimension does not follow from its presence in the approximation. As such we would need additional assertions in order to get the desired behavior.

\begin{question}
    Does the proposed notion of limit dimension for parameterized theories (Definition \ref{def:paramdim}) adequately generalize to examples outside of $T^*_{\feq}$? How can it be adapted to account for non-trivial $\acl$ in the base theory?
\end{question}

\begin{question}
    Is there a notion of limit dimension for $T_\infty$ which fits this axiomatic framework, and if so, does it recover Kim-independence?
\end{question}

\bibliography{biblio}
\bibliographystyle{plain}
\end{document}

%% file: preamble.tex
\usepackage{enumerate}
\usepackage{url}
\usepackage[font=small,labelfont=bf]{caption}
\usepackage[all,arc]{xy}
\usepackage{tabularx}
\usepackage{adjustbox}

\usepackage[a4paper,total={16cm,23.5cm},top=3cm,left=2.5cm]{geometry}

\usepackage{amssymb,latexsym}
\usepackage{amsmath,amsthm}
\usepackage{amsfonts,mathrsfs}
\usepackage{mathtools}

\usepackage{tikz}
\usetikzlibrary{shapes.geometric, arrows, positioning,decorations.pathreplacing,calligraphy}

\tikzstyle{bloc} = [rectangle, rounded corners, 
minimum width=3cm, 
minimum height=1cm,
text centered, 
draw=black, 
fill=airforceblue!30]

\tikzstyle{decision} = [diamond,
minimum width=3cm, 
minimum height=1cm, 
text centered, 
draw=black, 
fill=airforceblue!30]
\tikzstyle{arrow} = [thick,->,>=stealth]

\tikzstyle{io} = [trapezium, 
trapezium stretches=true, 
trapezium left angle=70, 
trapezium right angle=110, 
minimum width=3cm, 
minimum height=1cm, text centered, 
draw=black, fill=blue!30]

\tikzstyle{process} = [rectangle, 
minimum width=3cm, 
minimum height=1cm, 
text centered, 
text width=3cm, 
draw=black, 
fill=orange!30]

\usepackage{tikz-cd}
\usepackage{todonotes}

\usepackage[all]{xy}
\usepackage{graphicx}

\usepackage{hyperref}
\hypersetup{
    colorlinks,
    citecolor=blue,
    filecolor=blue,
    linkcolor=blue,
    urlcolor=blue
}

\usepackage{enumitem}

\usepackage{listings}
\usepackage{color}

\definecolor{dkgreen}{rgb}{0,0.6,0}
\definecolor{gray}{rgb}{0.5,0.5,0.5}
\definecolor{mauve}{rgb}{0.58,0,0.82}

\newtheorem{thm}{Theorem}[section]

\newtheorem{theorem}[thm]{Theorem}

\newtheorem{corollary}[thm]{Corollary}
\newtheorem{lemma}[thm]{Lemma}
\newtheorem{proposition}[thm]{Proposition}

\newtheorem*{theorem*}{Theorem}

\theoremstyle{definition}

\newtheorem{fact}[thm]{Fact}

\theoremstyle{definition}
\newtheorem{definition}[thm]{Definition}
\newtheorem{example}[thm]{Example}
\newtheorem{question}[thm]{Question}

\theoremstyle{remark}

\newtheorem{remark}[thm]{Remark}

\newtheorem*{pfclaim}{Claim}

\makeatletter
\let\c@equation\c@thm
\makeatother
\numberwithin{equation}{section}

\newcommand\nN{\mathbb{N}}

\def\lang{\mathcal{L}}

\def\cH{\mathcal{H}}
\def\M{\mathcal{M}}
\def\N{\mathcal{N}}
\def\cF{\mathcal{F}}

\def\Th{\operatorname{Th}}

\newcommand{\ACF}{\mathrm{ACF}}

\newcommand{\acl}{\mathrm{acl}}

\newcommand{\alg}{\mathrm{alg}}

\newcommand{\tp}{\mathrm{tp}}

\newcommand{\vect}[1]{\langle {#1} \rangle}
\newcommand{\abs}[1]{\lvert {#1} \rvert}

\newcommand{\Fraisse}{Fra\"issé }

\DeclareMathOperator{\im}{im}

\DeclareMathOperator{\Span}{span}
\DeclareMathOperator{\rk}{rk}
\DeclareMathOperator{\res}{res}

\definecolor{airforceblue}{rgb}{0.36, 0.54, 0.66}

\def\Ind{\setbox0=\hbox{$x$}\kern\wd0\hbox to 0pt{\hss$\mid$\hss}
\lower.9\ht0\hbox to 0pt{\hss$\smile$\hss}\kern\wd0}
\def\Notind{\setbox0=\hbox{$x$}\kern\wd0\hbox to 0pt{\mathchardef
\nn=12854\hss$\nn$\kern1.4\wd0\hss}\hbox to
0pt{\hss$\mid$\hss}\lower.9\ht0 \hbox to 0pt{\hss$\smile$\hss}\kern\wd0}
\def\ind{\mathop{\mathpalette\Ind{}}}
\def\nind{\mathop{\mathpalette\Notind{}}}

\def\indi#1{\mathop{\ \ \hbox to 0ex{\hss$\vert^{\hbox to 0ex{$\scriptstyle#1$\hss}}$\hss}
\lower1ex\hbox to 0ex{\hss$\smile$\hss}\ \ }}

\def\nindi#1{\mathop{\ \ \hbox to 0ex{\hss$\!\not{\vert}^{\hbox to 0ex{$\scriptstyle\,#1$\hss}}$\hss}
\lower1ex\hbox to 0ex{\hss$\smile$\hss}\ \ }}

\renewcommand{\models}{\vDash}

\DeclareMathOperator{\aut}{Aut}

\DeclareMathOperator{\feq}{feq}

\DeclareMathOperator{\gen}{gen}

\DeclareMathOperator{\sep}{sep}
\DeclareMathOperator{\nsop}{NSOP}

\DeclareMathOperator{\Sym}{Sym}
\DeclareMathOperator{\spann}{span}
\DeclareMathOperator{\Image}{Im}
\DeclareMathOperator{\GL}{GL}
\DeclareMathOperator{\SL}{SL}

\newcommand{\aaa}{\bar{a}}
\newcommand{\bb}{\bar{b}}
\newcommand{\cc}{\bar{c}}

\newcommand{\xx}{\bar{x}}
\newcommand{\vv}{\bar{v}}
\newcommand{\ww}{\bar{w}}
\newcommand{\yy}{\bar{y}}